\documentclass[11pt]{article}
\pdfoutput=1

\usepackage[margin=1.3in]{geometry}

\usepackage{graphicx}
\usepackage{amssymb,amsmath}
\usepackage{amsthm,enumerate}
\usepackage{hyperref,xcolor}
\usepackage{mathrsfs}
\usepackage{extarrows}
\usepackage{textcomp}

\makeatletter

\newdimen\bibspace
\makeatother

\numberwithin{equation}{section}

\newtheorem{theorem}{Theorem}[section]
\newtheorem{lemma}[theorem]{Lemma}
\newtheorem{proposition}[theorem]{Proposition}

\newtheorem{remark}[theorem]{Remark}

\def\<{\langle}
\def\>{\rangle}

\def\Xint#1{\mathchoice
{\XXint\displaystyle\textstyle{#1}}%
{\XXint\textstyle\scriptstyle{#1}}%
{\XXint\scriptstyle\scriptscriptstyle{#1}}%
{\XXint\scriptscriptstyle\scriptscriptstyle{#1}}%
\!\int}
\def\XXint#1#2#3{{\setbox0=\hbox{$#1{#2#3}{\int}$ }
\vcenter{\hbox{$#2#3$ }}\kern-.6\wd0}}

\def\fint{\Xint-}

\begin{document}

\title{Asymptotic behavior of positive solutions of some nonlinear elliptic equations on cylinders}

\author{Shan Chen, Zixiao Liu}
\date{\today}

\maketitle

\begin{abstract}
  We establish quantitative asymptotic behavior of positive solutions of a family of nonlinear elliptic equations on the half cylinder near the end.
  This unifies the study of isolated singularities of some semilinear elliptic equations, such as the Yamabe equation, Hardy-H\'enon equation etc.


  \textbf{2020~MSC:}~35J61,~ 35B40.
\end{abstract}

\section{Introduction}

Let $\mathbb S^{n-1}$ be the unit sphere in $\mathbb R^n$, $n\geq 3$.
In this paper, we study the asymptotic behavior of positive smooth solutions $u=u(\theta,t)$ of the
nonlinear elliptic equation
\begin{equation}\label{Equ-Cylinder}
\partial_{t t}^{2} u+\Delta_{\mathbb{S}^{n-1}}u+b\partial_tu+a u+u^{p}=0
\end{equation}
on $\mathbb{S}^{n-1} \times(0,+\infty),$
where $a,b\in\mathbb R, p>1$ are parameters and $\Delta_{\mathbb{S}^{n-1}}$ is the Laplace-Beltrami operator on $\mathbb S^{n-1}$. By  the inverse of Emden-Fowler transformation
\begin{equation}\label{Equ-transform}
v(x)=e^{\frac{n+b-2}{2} t} u(\theta, t),\quad r=e^{-t} \text { and } x=r \theta,
\end{equation}
\eqref{Equ-Cylinder} can be written as
\begin{equation}\label{Equ-Punctured}
\Delta v+\frac{c}{|x|^{2}} v+\frac{v^{p}}{|x|^{\sigma}}=0
\end{equation}
in $B_{1} \backslash\{0\}$ with
\begin{equation*}c=\frac{1}{4}\left((n-2)^{2}-b^{2}+4 a\right), \quad \sigma=2-\frac{1}{2}(p-1)(b+n-2),
\end{equation*}
where $B_r(x)$ is the open ball centered at $x$ with radius $r$  and $B_r:=B_r(0)$.

$\bullet $
When $c=0$, $\sigma=0$ and $p>1$, the isolated singularities of \eqref{Equ-Punctured} have been studied by Lions \cite{21} for $1<p<\frac{n}{n-2}$, Aviles \cite{3} for $p=\frac{n}{n-2}$, Gidas-Spruck \cite{GS} for $\frac{n}{n-2}<p<\frac{n+2}{n-2}$, Caffarelli-Gidas-Spruck
\cite{CGS} for $\frac{n}{n-2}<p\leq \frac{n+2}{n-2}$.  The critical case $p=\frac{n+2}{n-2}$ is  the Yamabe equation, which has been further studied by Li \cite{18}, Chen-Lin \cite{CCLin-CPAM97}, Korevaar-Mazzeo-Pacard-Schoen \cite{KMPS},  Li-Zhang \cite{Li-Zhang-Harnack},  Marques \cite{Marques-CVPDE},  Han-Li-Li \cite{han2019asymptotic},   Xiong-Zhang \cite{xiong-2020}  etc. If $n=3$ or $4$, \eqref{Equ-Punctured} is related to the stellar structure in astrophysics \cite{Book-n-equal-3} or the Yang-Mills equation \cite{GS} respectively.

$\bullet $
When $c=0,\sigma\in(-2,2)$ and $p>1$, the isolated singularities of \eqref{Equ-Punctured} have been studied by Gidas-Spruck \cite{GS} for $p\in (1,\frac{n-\sigma}{n-2})\cup(\frac{n-\sigma}{n-2},
\frac{n+2-\sigma}{n-2}
]$, Aviles \cite{4} for $p=\frac{n-\sigma}{n-2}$ and Phan-Souplet \cite{23} for $1<p<\frac{n+2}{n-2}$.

$\bullet $
When $c\not=0,\sigma=0$, the isolated singularities of \eqref{Equ-Punctured} have been studied by Bidaut-V\'eron-V\'eron \cite{5} for $1<p<\frac{n+2}{n-2}$ with any $c\in\mathbb R$,  Jin-Li-Xu \cite{15} for positive smooth solutions on entire $\mathbb R^n\setminus\{0\}$ with $p=\frac{n+2}{n-2}$, and recently by Chen-Zhou \cite{Chen-Feng-2018pub} for $1<p<\frac{n+2+\sqrt{(n-2)^{2}-4 c}}{n-2+\sqrt{(n-2)^{2}-4 c}}$ with $c \leq \frac{(n-2)^{2}}{4}$.

In this paper, we study the case both $c\not=0$ and $\sigma\not=0$, which arises from
the study of the Euler-Lagrange equation of the Gagliardo-Nirenberg interpolation inequalities and
Caffarelli-Kohn-Nirenberg inequalities, see the recent paper by Dolbeault-Esteban-Loss \cite{11}. We also note that if the sign of nonlinear term $\frac{u^p}{|x|^{\sigma}}$ in \eqref{Equ-Punctured} is negative, the phenomenon is quite different, see Br\'{e}zis-V\'{e}ron \cite{Brezis-1980}, V\'{e}ron \cite{Veron-singularities-new}, Bidaut-Veron-Grillot \cite{Veron-asymptoticbehavior},  C\^{\i}rstea \cite{Classification-FLor}, C\^{\i}rstea-F\u{a}rc\u{a}\c{s}eanu\cite{2020sharp} and the references therein.

The following is our first main theorem.
\begin{theorem}\label{Thm-1}
  Let $u $ be a positive classical solution of \eqref{Equ-Cylinder}  on
  $\mathbb{S}^{n-1} \times(0,+\infty)$ with
  \begin{equation}\label{Equ-cond-cylinder}
  b\geq 0,\quad
  b^2-4a\leq (n-2)^2\quad\text{and}\quad
  1<p\leq \dfrac{n+b+2}{n+b-2}.
\end{equation}
  Then
  \begin{equation*}\limsup _{t \rightarrow \infty} u(\theta, t)<\infty,
  \end{equation*}
  and
  \begin{equation*}\limsup _{t \rightarrow \infty}\left|\partial_{t} u(\theta, t)\right|<\infty, \quad \limsup _{t \rightarrow \infty}\left|\nabla_{\theta} u(\theta, t)\right|<\infty.
  \end{equation*}
  Moreover,
  \begin{equation}\label{equ-AsympRadial}
  u(\theta, t)=\overline{u}(t)\left(1+O\left(e^{-t}\right)\right)\quad\text {as } t \rightarrow \infty,
  \end{equation}
  where $\overline{u}(t)=\fint_{\mathbb{S}^{n-1}} u(\theta, t) d \theta$ is the spherical average of $u(\theta,t)$.
\end{theorem}

\begin{remark} If $b<0$, the  asymptotics of \eqref{Equ-Cylinder} have been studied by Simon \cite{27}.

By a direct computation,  condition \eqref{Equ-cond-cylinder} for \eqref{Equ-Cylinder} is equivalent to
\begin{equation}\label{Equ-cond-pucntured}
c\ge 0,\quad\sigma\in[0,2),\quad\text{and}\quad
1<p\leq \dfrac{n+2-2\sigma}{n-2}
\end{equation}
for \eqref{Equ-Punctured}.
When $\sigma=0$ and $p=\frac{n+2}{n-2}$, Jin-Li-Xu \cite{15} proved that when $c<-\frac{n-2}{4}$, \eqref{Equ-Punctured} has non-radial solutions on $\mathbb R^n\setminus\{0\}$. It is unknown whether there exists non-radial positive solutions when $-\frac{n-2}{4}\leq c<0$.
\end{remark}

By \eqref{equ-AsympRadial}, we have
\begin{equation}\label{Equ-averaged}
\overline{u}^{\prime \prime}+b \overline{u}^{\prime}+a \overline{u}+\overline{u}^{p}=-\fint_{\mathbb{S}^{n-1}}\left(u^{p}-\overline{u}^{p}\right)=\overline{u}^{p} O\left(e^{-t}\right),\quad 0<t<+\infty,
\end{equation}
which leads us to study the classification and asymptotics of  positive smooth solutions $\psi(t)$ of
\begin{equation}\label{Equ-averaged-homo}
\psi^{\prime \prime}(t)+b \psi^{\prime}(t)+a \psi(t)+\psi^{p}(t)=0, \quad-\infty<t<\infty.
\end{equation}
We separate \eqref{Equ-cond-cylinder} further  into the following three cases according to the difference on the asymptotics of positive smooth solutions of \eqref{Equ-averaged-homo} as $t\rightarrow+\infty$, see the appendix of the paper,
\begin{enumerate}[(i)]
  \item $b=0,a<0$;
  \item $b>0,a<0$;
  \item $a\geq 0$.
\end{enumerate}

Our second main result identify the limit of $u(\theta,t)$ at infinity.
\begin{theorem}\label{Thm-2}
  Let $u$ be a positive smooth solution of \eqref{Equ-Cylinder} and $a,b,p$ verifies \eqref{Equ-cond-cylinder}.
  \begin{enumerate}[(i)]
    \item\label{Case-1} When $b=0,a<0$, $u$ satisfies either
    \begin{equation}\label{Equ-case1-D0}
    \frac{1}{C}e^{-\sqrt{-a} t}\leq u(\theta, t)\leq Ce^{-\sqrt{-a} t},
    \end{equation}
    or
    \begin{equation}\label{Equ-case1-D1}
    \left|\dfrac{u(\theta, t)}{\psi(t)}-1\right|\leq Ce^{-\gamma t}
    \end{equation}
    for some $\gamma, C>0$
    as $t\rightarrow+\infty$,
    where $\psi(t)$ is a   positive periodic solution of \eqref{Equ-averaged-homo};
    \item\label{Case-2} When $b>0,a<0$, $u$ satisfies either
    \begin{equation}\label{equ-case2-D0}
    \frac{1}{C}e^{-\frac{b+\sqrt{b^2-4a}}{2} t}\leq u(\theta, t)\leq Ce^{-\frac{b+\sqrt{b^2-4a}}{2} t},
    \end{equation}
    or
    \begin{equation}\label{equ-case2-D1}
     \left|u(\theta, t)-(-a)^{\frac{1}{p-1}}\right|\leq Ce^{-\gamma t}
    \end{equation}
    for some $\gamma,C>0$ as $t \rightarrow +\infty.$
    \item\label{Case-3} When $a\geq 0,$ if $a=0$ or $b=0$, then
    $$
    u(\theta,t)=o(1)
    $$
    as $t\rightarrow+\infty$.  If $b>0,a>0$, then
    \begin{equation}\label{equ-case3-D}
    \frac{1}{C}e^{-\gamma_1t}\leq u(\theta,t)\leq Ce^{-\gamma_2t}
    \end{equation}
    for some $\gamma_1,\gamma_2,C>0$
    as $t\rightarrow \infty$.
  \end{enumerate}
\end{theorem}
\begin{remark}\label{Rem-1}
  Note that when $c>0$, if $u$ satisfies \eqref{Equ-case1-D0} or \eqref{equ-case2-D0}, then the corresponding $v$ from \eqref{Equ-transform} belongs to $H^1_{loc}(B_1)$ but is not bounded near origin.
\end{remark}
The paper is organized as follows. In section \ref{Sec-2}, we prove Theorem \ref{Thm-1} by blow-up analysis and the method of moving spheres. In section \ref{Sec-3}, we prove Theorem \ref{Thm-2} by ODE analysis. In section \ref{Sec-4}, we discuss whether the singularity of $v(x)$ at origin is removable. In the Appendix we provide the classification result on the asymptotics of  \eqref{Equ-averaged-homo}.

\section{Upper bound and asymptotic symmetry}\label{Sec-2}

By transform \eqref{Equ-transform}, Theorem \ref{Thm-1} is equivalent to the following.
\begin{theorem}\label{Thm-1-Punctured}
 Let $v$ be a positive classical solution of \eqref{Equ-Punctured} in $B_1\setminus\{0\}$ with \eqref{Equ-cond-pucntured}. Then
 \begin{equation}\label{equ-Thm-bound-subcri}
 \limsup _{x \rightarrow 0}|x|^{\frac{2-\sigma}{p-1}} v(x)<\infty,\quad\limsup _{x \rightarrow 0}|x|^{\frac{p+1-\sigma}{p-1}}|\nabla v(x)|<\infty.
 \end{equation}
 Moreover,
 \begin{equation}\label{Equ-spherical-average}
 v(x)=\overline{v}(|x|)(1+O(|x|))\quad\text {as } x \rightarrow 0,
 \end{equation}
 where $\overline v(|x|)=\fint_{\mathbb S^{n-1}}v(|x|\theta)d\theta$.
\end{theorem}

In   subsection \ref{Sec-2.1}, we prove \eqref{equ-Thm-bound-subcri} for $1<p<\frac{n+2}{n-2}$ case. In   subsection \ref{Sec-2.2}, we prove \eqref{equ-Thm-bound-subcri} for $p=\frac{n+2}{n-2}$ case. In subsection \ref{Sec-2.3} we prove that \eqref{Equ-spherical-average} follows from \eqref{equ-Thm-bound-subcri}.

\subsection{Upper bound when $p<\frac{n+2}{n-2}$}\label{Sec-2.1}

The proof relies on the blow-up analysis as Theorem 1.2 in \cite{23} and Theorem 2.1 in \cite{xiong-2020}.
\begin{lemma}\label{Lem-Blowup-subcritical}
Let $1<p<\frac{n+2}{n-2}$ and $\gamma_{i} \in(0,1],~ i=1,2 .$ Let $l \in C^{\gamma_{1}}\left(\overline{B}_{1}\right)$ and
$h \in C^{\gamma_{2}}\left(\overline{B}_{1}\right)$ satisfy
\begin{equation}\label{Equ-bound-temp-1}
\|h\|_{C^{\gamma_{2}}\left(\overline{B}_{1}\right)} \leq C_{1}, \quad\|l\|_{C^{\gamma_{1}}\left(\overline{B}_{1}\right)} \leq C_{1}^{\prime},
\end{equation}
and
\begin{equation}\label{Equ-bound-temp-2}
h(x)\geq C_2,\quad x\in\overline B_1
\end{equation}
for some constants $C_1,C_1',C_2>0$. Then there exists $C=C(\gamma_1,\gamma_2,C_1,C_1',C_2,p,n)>0$ such that  any nonnegative classical solution $v(x)$ of
\begin{equation}\label{Equ-equ-temp}
-\Delta v=l(x) v+h(x) v^{p}, \quad x \in B_{1},
\end{equation}
satisfies
\begin{equation*}|v(x)|^{\frac{p-1}{2}}+|\nabla v(x)|^{\frac{p-1}{p+1}} \leq C\left(1+\mathtt{dist}^{-1}\left(x, \partial B_{1}\right)\right), \quad x \in B_{1},
\end{equation*}
where $\mathtt{dist}(x,\partial B_1)$  is the distance between $x$ and $\partial B_1$.
\end{lemma}
\begin{proof}
  By contradiction, we assume that there exist sequences $\left\{l_{k}\right\}_{k=1}^{\infty},\left\{h_{k}\right\}_{k=1}^{\infty}$ and $\left\{v_{k}\right\}_{k=1}^{\infty}$ satisfying \eqref{Equ-bound-temp-1}, \eqref{Equ-bound-temp-2} and \eqref{Equ-equ-temp} and $\{y_k\}_{k=1}^{\infty}\subset B_1$ such that
  \begin{equation*}M_{k}(x):=\left|v_{k}(x)\right|^{\frac{p-1}{2}}+\left|\nabla v_{k}(x)\right|^{\frac{p-1}{p+1}},\quad x \in B_{1},
  \end{equation*}
  satisfy
  \begin{equation}\label{equ-temp-1}
  M_{k}\left(y_{k}\right)>2 k\left(1+\mathtt{dist}^{-1}\left(y_{k}, \partial B_{1}\right)\right) \geq 2 k \mathtt{ dist }^{-1}\left(y_{k}, \partial B_{1}\right).\end{equation}
  By Lemma 5.1 in \cite{24}, there exists $x_k\in B_1$ such that
  \begin{equation*}M_{k}\left(x_{k}\right) \geq M_{k}\left(y_{k}\right), \quad M_{k}\left(x_{k}\right)>2 k \mathtt{dist}^{-1}\left(x_{k}, \partial B_{1}\right),
  \end{equation*}
  and
  \begin{equation}\label{equ-temp-2}
  M_{k}(z) \leq 2 M_{k}\left(x_{k}\right)\quad \text{for all}~\left|z-x_{k}\right| \leq k M_{k}^{-1}\left(x_{k}\right).
  \end{equation}
  Let $\lambda_k:=M_k^{-1}(x_k)~\text{and}\quad$
  $$
w_{k}(y):=\lambda_{k}^{\frac{2}{p-1}} v_{k}\left(x_{k}+\lambda_{k} y\right).
  $$
  By a direct computation,
  \begin{equation}\label{equ-temp-3}
  -\Delta w_{k}=\lambda_{k}^{2} \tilde{l}_{k}(y) w_{k}+\tilde{h}_{k}(y) w_{k}^{p}, \quad|y| \leq k,\end{equation}
where $\tilde{l}_{k}(y)=l_{k}\left(x_{k}+\lambda_{k} y\right)$ and $\tilde{h}_{k}(y)=h_{k}\left(x_{k}+\lambda_{k} y\right)$.  Furthermore, by \eqref{equ-temp-1} and \eqref{equ-temp-2}, $\lambda_k\leq \frac{1}{2k}$, $|w_k(0)|^{\frac{p-1}{2}}+|\nabla w_k(0)|^{\frac{p-1}{p+1}}=1$ and
\begin{equation}\label{equ-temp-4}\left[\left|w_{k}\right|^{\frac{p-1}{2}}+\left|\nabla w_{k}\right|^{\frac{p-1}{p+1}}\right](y)=\frac{M_{k}\left(x_{k}+\lambda_{k} y\right)}{M_{k}\left(x_{k}\right)} \leq 2,\quad |y| \leq k.\end{equation}
By  \eqref{Equ-bound-temp-1} and \eqref{Equ-bound-temp-2}, we have $$[\tilde h_k]_{C^{\gamma_2}(B_k)}\leq \lambda_k^{\gamma_2}C_1,~~||\lambda_k^2\tilde l_k||_{C^{\gamma_1}(B_k)}\leq \lambda_k^2C_1'\quad\text{and}\quad
C_1\geq \tilde h_k\geq C_2.$$
By Arzela-Ascoli theorem, there exist a subsequence (still denoted as $\tilde h_k$) and $\tilde h\in C(\mathbb R^n)$ such that
$
\tilde h_k\rightarrow \tilde h~\text{in}~C_{loc}(\mathbb R^n).
$
Furthermore, for any $x,y\in\mathbb R^n$,
$$
|\tilde h(x)-\tilde h(y)|
=\lim_{k\rightarrow+\infty}|\tilde h_k(x)-\tilde h_k(y)|
\leq \limsup_{k\rightarrow+\infty}\lambda_k^{\gamma_2}C_1|x-y|^{\gamma_2}=0.
$$
Thus
$\tilde h\equiv H$ for some constant $H\in [C_2,C_1]$.

For any $R>0, n<q<\infty$, by \eqref{equ-temp-3}, \eqref{equ-temp-4} and interior $L^q$-estimates, $w_k$ is uniformly bounded in $W^{2,q}(B_R)$. By Morrey's inequality, compact embedding and extracting subsequences, there exists $w\in C^2(\mathbb R^n)$ such that $$w_{k} \rightarrow w\quad\text {in } C_{l o c}^{2}\left(\mathbb{R}^{n}\right)$$
and
\begin{equation}\label{equ-temp-5}
-\Delta w=H w^{p}\quad\text {in } \mathbb{R}^{n}.
\end{equation}
Since $w_k$ are positive functions satisfying $\left|w_{k}(0)\right|^{\frac{p-1}{2}}+\left|\nabla w_{k}(0)\right|^{\frac{p-1}{p+1}}=1$, we have $w\geq 0$ and $|w|^{\frac{p-1}{2}}(0)+|\nabla w|^{\frac{p-1}{p+1}}(0)=1$. Thus $w$ must be a positive $C^2$ solution of \eqref{equ-temp-5}, contradicting to the Liouville-type theorem as Theorem 1.1 in \cite{GS}.
\end{proof}

Now we prove the upper bound in Theorem \ref{Thm-1-Punctured} for $p<\frac{n+2}{n-2}$ case.
\begin{proposition}\label{Prop-2.3}
  Let $v$ be a positive classical solution of \eqref{Equ-Punctured} in $B_1\setminus\{0\}$ with \eqref{Equ-cond-pucntured} and $p<\frac{n+2}{n-2}$. Then \eqref{equ-Thm-bound-subcri} holds.
\end{proposition}
\begin{proof}
  It suffices to prove that there exists a positive constant $C$ such that for any  $0<\left|x_{0}\right|<\frac{1}{2},$ $$\left|x_{0}\right|^{\frac{2-\sigma}{p-1}} v\left(x_{0}\right) \leq C\quad\text{and}\quad\left|x_{0}\right|^{\frac{p+1-\sigma}{p-1}}\left|\nabla v\left(x_{0}\right)\right| \leq C.$$
  Let $R:=\frac{1}{2}|x_0|$ and
  \begin{equation*}V(y):=R^{\frac{2-\sigma}{p-1}} v\left(x_{0}+R y\right),\quad y\in B_1.
  \end{equation*}
  Then $V$ is a positive classical solution of
  \begin{equation*}
  -\Delta V=l(y) V+h(y) V^{p}\quad\text {in } B_{1},
  \end{equation*}
  where $l(y)=\frac{c}{\left|\frac{x_{0}}{R}+y\right|^{2}}, h(y)=\frac{1}{\left|\frac{x_{0}}{R}+y\right|^{\sigma}}$. Since
  $1 \leq\left|\frac{x_{0}}{R}+y\right| \leq 3$ for all $y \in \overline{B}_{1}$, $h$ and $l$ satisfy conditions \eqref{Equ-bound-temp-1} and \eqref{Equ-bound-temp-2}. By Lemma \ref{Lem-Blowup-subcritical},
  \begin{equation*}|V(0)|^{\frac{p-1}{2}}+|\nabla V(0)|^{\frac{p-1}{p+1}} \leq C,
  \end{equation*}
  for some constant $C$ independent of $x_0$. Then \eqref{equ-Thm-bound-subcri} follows immediately.
\end{proof}
\begin{remark}
  The proof for $p=\frac{n+2}{n-2}$ case is very different from the proof for $1<p<\frac{n+2}{n-2}$. This is due to $-\Delta w=w^{\frac{n+2}{n-2}}$ in $\mathbb R^n\setminus\{0\}$ has non-trivial positive solutions \cite{CGS}.
\end{remark}

\subsection{Upper bound when $p=\frac{n+2}{n-2}$}\label{Sec-2.2}

 The proof relies on blow-up analysis and the method of moving spheres, similar to the proof in \cite{19,Li-FracHardy,xiong-2020}. For any $x_0\in\mathbb R^n$ and $\lambda>0$, we let
 \begin{equation*}w_{x_{0}, \lambda}(x):=\left(\frac{\lambda}{\left|x-x_{0}\right|}\right)^{n-2} w\left(x_{0}+\frac{\lambda^{2}\left(x-x_{0}\right)}{\left|x-x_{0}\right|^{2}}\right)
 \end{equation*}
 denote the Kelvin transform of $w$ with respect to $\partial B_{\lambda}(x_0)$. By \eqref{Equ-cond-pucntured}, if $\sigma>0$ then $p\leq \frac{n+2-2\sigma}{n-2}<\frac{n+2}{n-2}$. Thus we have $\sigma=0$ in this case.

\begin{proposition}\label{Prop-2.5}
  Let $v$ be a positive classical solution of \eqref{Equ-Punctured} in $B_1\setminus\{0\}$ with \eqref{Equ-cond-pucntured}, $\sigma=0$ and $p=\frac{n+2}{n-2}$. Then \eqref{equ-Thm-bound-subcri} holds.
\end{proposition}
\begin{proof}

  Without loss of generality, we assume that $v$ is continuous to the boundary $\partial B_{1}$ and $v>\frac{1}{M}$ on $\partial B_{1}$ for some $M>0 .$ Otherwise we can consider the equation in $B_{3 / 4} \backslash\{0\}$.  By contradiction, we assume that
  there exists a sequence $\left\{x_{j}\right\}_{j=1}^{\infty} \subset B_{1}$ such that $x_j\rightarrow 0$ and
  \begin{equation}\label{equ-temp-8}
  \left|x_{j}\right|^{\frac{n-2}{2}} v\left(x_{j}\right) \rightarrow \infty\quad \text {as } j \rightarrow \infty.
  \end{equation}
  Let
  \begin{equation*}\omega_{j}(x):=\left(\frac{\left|x_{j}\right|}{2}-\left|x-x_{j}\right|\right)^{\frac{n-2}{2}} v(x), \quad\left|x-x_{j}\right| \leq \frac{\left|x_{j}\right|}{2}.
  \end{equation*}
  Then $\omega_j(x)$ admits at least one interior maximum point, denoted as $\overline x_j$ such that
  \begin{equation*}\omega_{j}\left(\overline{x}_{j}\right)=\max _{\left|x-x_{j}\right| \leq \frac{|x_j|}{2}} \omega_{j}(x)>0.
  \end{equation*}
  Let $2 \mu_{j}:=\frac{\left|x_{j}\right|}{2}-\left|\overline{x}_{j}-x_{j}\right|$, then by triangle inequality,
  \begin{equation*}0<2 \mu_{j} \leq \frac{\left|x_{j}\right|}{2}\quad \text {and}\quad \frac{\left|x_{j}\right|}{2}-\left|x-x_{j}\right| \geq \mu_{j},\quad \forall~\left|x-\overline{x}_{j}\right| \leq \mu_{j}.\end{equation*}
  Together with the definition of $\omega_j$ and $\mu_j$,
  \begin{equation*}\left(2 \mu_{j}\right)^{\frac{n-2}{2}} v\left(\overline{x}_{j}\right)=\omega_{j}\left(\overline{x}_{j}\right) \geq \omega_{j}(x) \geq\left(\mu_{j}\right)^{\frac{n-2}{2}} v(x),\quad~ \forall~\left|x-\overline{x}_{j}\right| \leq \mu_{j},
  \end{equation*}
  i.e.,
  \begin{equation}\label{equ-temp-6}
  2^{\frac{n-2}{2}} v\left(\overline{x}_{j}\right) \geq v(x),\quad\forall~\left|x-\overline{x}_{j}\right| \leq \mu_{j}.
  \end{equation}
  From \eqref{equ-temp-8}, we also have
  \begin{equation}\label{equ-temp-7}
  \left(2 \mu_{j}\right)^{\frac{n-2}{2}} v\left(\overline{x}_{j}\right)=\omega_{j}\left(\overline{x}_{j}\right) \geq \omega_{j}\left(x_{j}\right)=\left(\frac{\left|x_{j}\right|}{2}\right)^{\frac{n-2}{2}} v\left(x_{j}\right) \rightarrow \infty \quad \text {as } j \rightarrow \infty.
  \end{equation}
  Now, consider
  \begin{equation*}w_{j}(y):=\frac{1}{v\left(\overline{x}_{j}\right)} v\left(\overline{x}_{j}+\frac{y}{v\left(\overline{x}_{j}\right)^{\frac{2}{n-2}}}\right), \quad y \in \Omega_{j}:=\left\{y \in \mathbb{R}^{n} \mid \overline{x}_{j}+\frac{y}{v\left(\overline{x}_{j}\right)^{\frac{2}{n-2}}} \in B_{1} \backslash\{0\}\right\}.
  \end{equation*}
  Then $w_j(0)=1$ and
  \begin{equation*}-\Delta w_{j}=w_{j}^{\frac{n+2}{n-2}}+c_{j}(y) w_{j} \quad \text {in } \Omega_{j},
  \end{equation*}
  where
  \begin{equation*}c_{j}(y)=\frac{c}{\left|y+\overline{x}_{j} v\left(\overline{x}_{j}\right)^{\frac{2}{n-2}}\right|^{2}}.\end{equation*}
  It follows from \eqref{equ-temp-6} and \eqref{equ-temp-7} that
  \begin{equation*}w_{j}(y) \leq 2^{\frac{n-2}{2}}\quad\text {in } B_{R_{j}},
  \end{equation*}
  where $R_{j}:=\mu_{j} v\left(\overline{x}_{j}\right)^{\frac{2}{n-2}}\rightarrow+\infty$ as $j\rightarrow+\infty.$ Furthermore, for all $y\in B_{R_j}$,
  \begin{equation*}0<\frac{c}{\left|\frac{7}{4}| x_{j}|v\left(\overline{x}_{j}\right)^{\frac{2}{n-2}}\right|^{2}} \leq c_{j}(y) \leq \frac{c}{\left|\mu_{j} v\left(\overline{x}_{j}\right)^{\frac{2}{n-2}}\right|^{2}}\rightarrow 0\quad\text{as}~j\rightarrow\infty,
  \end{equation*}
  and $[c_j]_{C^{0,1}(B_{R_j})}$ is uniformly (to $j$) bounded. By interior $L^p$-estimates and   embedding theories again, there exists a subsequence (still denoted as $w_j$) such that
  $w_{j} \rightarrow w  $ in $C_{l o c}^{2}\left(\mathbb{R}^{n}\right)$
  for some $w\geq 0$ satisfying
  \begin{equation}\label{Equ-extremal}
  -\Delta w=
  w^{\frac{n+2}{n-2}}\quad \text {in } \mathbb{R}^{n}.
  \end{equation}
  Since $w(0)=1$ and $\nabla w(0)=0$, by the Liouville type theorem of Caffarelli-Gidas-Spruck \cite{CGS}, $w=\left(\frac{n(n-2)}{n(n-2)+|y|^2}\right)^{\frac{n-2}{2}}$.

  On the other hand, we   prove that $w$ has to be a constant. To that end, for any $x_0\in\mathbb R^n,\lambda_0>0$, let $j$ sufficiently large such that $\left|x_{0}\right|<\frac{R_{j}}{10}, \lambda_{0}<\frac{R_{j}}{10}$.

  \textit{Claim 1.} There exists $\lambda_1>0$ such that
  \begin{equation*}\left(w_{j}\right)_{x_{0}, \lambda}(y) \leq w_{j}(y), \quad \forall~ 0<\lambda<\lambda_{1},~ y \in \Omega_{j} \backslash B_{\lambda}\left(x_{0}\right).
  \end{equation*}
  By a direct computation, since $w_j>0$ is smooth, there exists sufficiently small $r_{x_0}$ such that
  \begin{equation}\label{equ-stating}
\frac{\mathrm{d}}{\mathrm{d} r}\left(r^{\frac{n-2}{2}} w_{j}\left(x_{0}+r \theta\right)\right)>0
\end{equation}
for all $0<r\leq r_{x_0},~\theta\in\mathbb S^{n-1}$.  Thus for any $y \in B_{r_{x_{0}}}\left(x_{0}\right), 0<\lambda<\left|y-x_{0}\right| \leq r_{x_{0}},$ let $\theta=\frac{y-x_{0}}{|y-x_{0}|}, r_{1}=\left|y-x_{0}\right|, r_{2}=\frac{\lambda^{2}}{|y-x_0|^{2}} r_{1},$ we have $r_{2} <r_{1}$ and
$$
r_2^{\frac{n-2}{2}}w_j(x_0+r_2\theta)<r_1^{\frac{n-2}{2}} w_j(x_0+r_1\theta).
$$
Namely,
\begin{equation*}\left(w_{j}\right)_{x_{0}, \lambda}(y) \leq w_{j}(y), \quad\forall~ 0<\lambda<\left|y-x_{0}\right| \leq r_{x_{0}}.\end{equation*}
Let
\begin{equation*}\phi_{j}(y):
=\left(\frac{r_{x_{0}}}{\left|y-x_{0}\right|}\right)^{n-2} \inf _{\partial B_{r_{x_{0}}}\left(x_{0}\right)} w_{j},
\end{equation*}
which is harmonic away from $x_0$. By $c\geq 0$,
\begin{equation}\label{equ-temp-equ}
\Delta\left(w_{j}-\phi_{j}\right) \leq 0\quad\text {in } \Omega_{j} \backslash B_{r_{x_{0}}}\left(x_{0}\right).
\end{equation}
By the definition of $\phi_j$,
\begin{equation}\label{equ-temp-innbound}
w_{j}-\phi_{j} \geq 0\quad\text{on}~\partial B_{r_{x_{0}}}\left(x_{0}\right).
\end{equation}
Now we examine the other boundary $y\in\partial\overline{\Omega_j}$, i.e., $\left|\overline{x}_{j}+\frac{y}{v\left(\overline{x}_{j}\right)^{\frac{2}{n-2}}}\right|=1$. By the lower bound $v>\frac{1}{M}$ on $\partial B_1$, we have
\begin{equation}\label{equ-temp-bound-1}
w_{j} > \frac{1}{M v\left(\overline{x}_{j}\right)}>0\quad\text {on } \partial \overline{\Omega}_{j}.
\end{equation}
Since $\frac{\left|x_{j}\right|}{2} \leq\left|\overline{x}_{j}\right| \leq \frac{3\left|x_{j}\right|}{2} \ll 1,$ for any $y \in \partial \overline{\Omega}_{j}$, we have $|y| \approx v\left(\overline{x}_{j}\right)^{\frac{2}{n-2}}$. Therefore for sufficiently large $j$,
\begin{equation}\label{equ-temp-bound-2}
w_{j}(y)> \frac{1}{M v\left(\overline{x}_{j}\right)} \approx \frac{1}{M|y|^{\frac{n-2}{2}}}>\left(\frac{r_{x_{0}}}{\left|y-x_{0}\right|}\right)^{n-2} \inf _{\partial B_{r_{0}}\left(x_{0}\right)} w_{j}=\phi_j(y)\quad\text {on } \partial \overline{\Omega}_{j}.\end{equation}
Combining \eqref{equ-temp-equ}, \eqref{equ-temp-innbound} and \eqref{equ-temp-bound-2}, by maximum principle as Lemma 2.1 in \cite{CCLin-Duke95} we have $$w_{j} \geq \phi_{j}\quad\text{in}~\Omega_{j} \backslash B_{r_{x_0}}\left(x_{0}\right).$$Let $$\lambda_{1}:=r_{x_{0}}\left(\frac{\inf _{\partial B_{r_{0}}\left(x_{0}\right)} w_{j}}{\sup _{B_{r_{x_{0}}}\left(x_{0}\right)} w_{j}}\right)^{\frac{1}{n-2}}.$$
Then for any $0<\lambda<\lambda_{1},\left|y-x_{0}\right| \geq r_{x_{0}}$ and $y \in \Omega_{j},$ we have
\begin{equation}\label{equ-2xx}
\begin{aligned}
\left(w_{j}\right)_{x_{0}, \lambda}(y) &=\left(\frac{\lambda}{\left|y-x_{0}\right|}\right)^{n-2} w_{j}\left(x_{0}+\frac{\lambda^{2}\left(y-x_{0}\right)}{\left|y-x_{0}\right|^{2}}\right) \\
& \leq\left(\frac{\lambda_{1}}{\left|y-x_{0}\right|}\right)^{n-2} \sup _{B_{r_{x_{0}}}\left(x_{0}\right)} w_{j} \\
&=\left(\frac{r_{x_{0}}}{\left|y-x_{0}\right|}\right)^{n-2} \inf _{\partial B_{r_{x_{0}}}\left(x_{0}\right)} w_{j} \\
& < w_{j}(y).
\end{aligned}\end{equation}
This finishes the proof of \textit{Claim 1}.

We define
\begin{equation*}\overline{\lambda}(x_0):=\sup \left\{0<\mu \leq \lambda_{0}\left|\left(w_{j}\right)_{x_{0}, \lambda}(y) \leq w_{j}(y),~~ \forall~\right. |y-x_{0}| \geq \lambda,~ y \in \Omega_{j}, ~ 0<\lambda<\mu\right\},
\end{equation*}
where $\lambda_{0}$ and $x_{0}$ are fixed at the beginning. By  \textit{Claim 1}, $\overline{\lambda}(x_0)$ is well defined.

\textit{Claim 2.} $\overline\lambda(x_0)=\lambda_0$ for any $x_0\in\mathbb R^n$.

By contradiction, we assume that for some $x_0\in\mathbb R^n$, $\overline\lambda(x_0)<\lambda_0$. By the continuity of $w_j$, $(w_j)_{x_0,\overline\lambda}(y)\leq w_j(y)$ for all $y\in\Omega_j\setminus B_{\overline\lambda}(x_0)$. By a direct computation,
\begin{equation}\label{equ-maximumP}
\begin{array}{lllll}
-\Delta\left(w_{j}-\left(w_{j}\right)_{x_{0}, \overline{\lambda}}\right)&=&\displaystyle\left(w_{j}^{n+2}-\left(w_{j}\right)_{x_{0}, \overline{\lambda}}^{n+2}\right)\\
&&+\displaystyle\left(c_{j}(y) w_{j}-\left(\frac{\overline{\lambda}}{\left|y-x_{0}\right|}\right)^{4} c_{j}\left(x_{0}+\frac{\overline{\lambda}^{2}\left(y-x_{0}\right)}{\left|y-x_{0}\right|^{2}}\right)\left(w_{j}\right)_{x_{0}, \overline{\lambda}}\right)\\
&\geq& 0\\
\end{array}
\end{equation}
in $\Omega_{j} \backslash\left\{B_{\overline{\lambda}}\left(x_{0}\right) \cup\left\{P\right\}\right\}$, where $P$ is the singular point of $(w_j)_{x_0,\lambda}$.  As in the proof of \eqref{equ-temp-bound-2}, for $y\in\partial\overline{\Omega_j}$ and $j$ sufficiently large,
\begin{equation*}\begin{aligned}
\left(w_{j}\right)_{x_{0}, \overline{\lambda}}
& \leq\left(\frac{\overline{\lambda}}{\left|y-x_{0}\right|}\right)^{n-2} \sup _{B_{\overline{\lambda}}\left(x_{0}\right)} w_{j} \\
& \leq\left(\frac{\overline{\lambda}}{\left|y-x_{0}\right|}\right)^{n-2} 2^{\frac{n-2}{2}} \approx \frac{1}{v\left(\overline{x}_{j}\right)^{2}}\\
&<w_j(y).\\
\end{aligned}\end{equation*}
By maximum principle, Hopf lemma and the fact that $w_j-(w_j)_{x_0,\overline\lambda}=0$ on $\partial B_{\overline\lambda}(x_0)$,
we have $w_j-(w_j)_{x_0,\overline\lambda}>0$ in $\Omega_j\setminus\overline{B_{\overline{\lambda}}(x_0)}$ and $\frac{\partial}{\partial \mathcal V}(w_j-(w_j)_{x_0,\overline\lambda})>0$ on $\partial B_{\overline{\lambda}}(x_0)$, where $\mathcal V$ is the unit exterior normal vector of $B_{\overline{\lambda}}(x_0)$.
By a standard argument in moving spheres method as in \cite{15,xiong-2020}, we can move spheres a little further than $\overline\lambda$, contradicting to the definition of $\overline\lambda$. Therefore \textit{Claim 2} is proved.

Sending $j\rightarrow+\infty$, we obtain
\begin{equation*}w_{x, \lambda}(y) \leq w(y), \quad \forall~ x \in \mathbb{R}^{n}, \lambda>0,~|y-x| \geq \lambda.
\end{equation*}
Thus by Lemma 2.1 in \cite{15} or Lemma 11.2 in \cite{Li-Zhang-Harnack}, $w$ is a constant, contradicting to  $w=\left(\frac{n(n-2)}{n(n-2)+|y|^{2}}\right)^{\frac{n-2}{2}}$ from \eqref{Equ-extremal}.

Hence there exists $C>0$ such that  $v(x) \leq C|x|^{-\frac{n-2}{2}}$  for all $x\in B_1\setminus\{0\}$. It remains to prove the upper bound of $|\nabla v(x)|$.
For all $0<r<\frac{1}{4}$, let
\begin{equation*}v_{r}(y):=r^{\frac{n-2}{2}} v(r y),\quad
y\in B_{\frac{1}{r}}.
\end{equation*}
By a direct computation, we have $v_r(y)$ is uniformly (to $r$) bounded in $\frac{1}{2} \leq|y| \leq \frac{3}{2}$ and satisfies
\begin{equation*}\Delta v_{r}(y)=-\frac{c}{|y|^{2}} v_{r}(y)-v_{r}^{\frac{n+2}{n-2}}(y)\quad \text{in}~ \frac{1}{2} \leq|y| \leq \frac{3}{2}.
\end{equation*}
By Harnack inequality, interior $L^q$-estimates and Morrey's inequality,
$$\left|\nabla v_{r}\right| \leq C v_{r},\quad\forall~|y|=1.$$ Thus $|\nabla v(x)| \leq C \frac{v(x)}{|x|} \leq C|x|^{-\frac{n}{2}}$ and the proof is finished.
\end{proof}

\subsection{Asymptotic symmetry}\label{Sec-2.3}

\begin{proposition}\label{prop-2.7}
  Let $v$ be a positive classical solution of \eqref{Equ-Punctured} in $B_1\setminus\{0\}$ with \eqref{Equ-cond-pucntured}. Then \eqref{Equ-spherical-average} holds.
\end{proposition}
The proof is based on moving spheres method again, similar to the proof of Theorem 1.3 in \cite{19}. Such method provides the convergence to averaged integral with a rate of $O(|x|)$.
\begin{proof}
  By Propositions \ref{Prop-2.3} and \ref{Prop-2.5}, we have \eqref{equ-Thm-bound-subcri}.
  As in the proof of Proposition \ref{Prop-2.5}, we may assume that $v \geq \frac{1}{M}$ on $\partial B_{1}$ for some $M>0$. We claim that there exists $0<\epsilon<\frac{1}{10}$ such that for any $0<|x| \leq \epsilon$,
  \begin{equation}\label{Claim}
  v_{x, \lambda}(y) \leq v(y),\quad\forall~ 0<\lambda<|x|,~\lambda \leq|y-x|.
  \end{equation}
  As in \eqref{equ-stating}, for any given $x$, there exists sufficiently small $0<r_x<|x|$ such that
  \begin{equation}\label{equ-temp-10}
  v_{x, \lambda}(y) \leq v(y),\quad\forall~ 0<\lambda<r_{x},~ \lambda \leq|y-x| \leq r_{x}.
  \end{equation}
  Let $0<\lambda_{1} \leq r_{x}\left(\inf _{\partial B_{r_x}(x)} v / \sup _{B_{r_{x}}(x)} v\right)^{\frac{1}{n-2}} \leq r_{x}$ and $r_x$ even smaller such that for all $y\in \partial B_1$,
  \begin{equation}\label{Equ-temp-9}
  v(y)\geq \left(\dfrac{r_x}{|y-x|}\right)^{n-2}\inf_{\partial B_{r_x}(x)}v.
  \end{equation}
  Clearly \eqref{Equ-temp-9} also holds for $y\in \partial B_{r_x}(x)$. By maximum principle,
  $$
  v(y)\geq \left(\dfrac{r_x}{|y-x|}\right)^{n-2}\inf_{\partial B_{r_x}(x)}v,\quad\forall~y\in \overline{B_1\setminus B_{r_x}(x)}.
  $$
  As in \eqref{equ-2xx}, it follows that
  \begin{equation}\label{equ-temp-11}
  v_{x,\lambda}(y)\leq v(y),\quad\forall~ 0<\lambda\leq\lambda_1,~|y-x|\geq r_x.
  \end{equation}
Combining \eqref{equ-temp-10} and \eqref{equ-temp-11}, we see that
\begin{equation*}\overline{\lambda}(x):=\sup \left\{0<\mu \leq|x|: v_{x,\lambda}(y) \leq v(y),~\forall~0<\lambda<\mu,~y \in B_{1} \backslash B_{\lambda}(x)\right\}
\end{equation*}
is well-defined.

It remains to prove that there exists $\epsilon>0$ independent of $x$ such that $\overline\lambda(x)=|x|$ for all $0<|x|\leq \epsilon$. By contradiction, we assume that $\overline\lambda(x)<|x|$ for some $x$.

By moving spheres method, we only need to focus on maintaining $v_{x,\overline \lambda}<v$ on the exterior boundary $y\in\partial B_1$.
For any $0<\lambda<|x|$ and   $y \in \partial B_{1}$,
\begin{equation*}\left|x+\frac{ \lambda^{2}(y-x)}{|y-x|^{2}}\right| \geq|x|- \lambda^{2} \frac{1}{|y-x|} \geq|x|-\frac{10}{9} \lambda^{2} \geq|x|-\frac{10}{9}|x|^{2} \geq \frac{8}{9}|x|.
\end{equation*}
Together with \eqref{equ-Thm-bound-subcri}, we have
\begin{equation*}v\left(x+\frac{ \lambda^{2}(y-x)}{|y-x|^{2}}\right) \leq C\left|x+\frac{ \lambda^{2}(y-x)}{|y-x|^{2}}\right|^{-\frac{2-\sigma}{1-p}} \leq C|x|^{-\frac{2-\sigma}{1-p}}
\end{equation*}
for some $C>0$ independent of $x$. Thus for $y\in\partial B_1$,
\begin{equation*}\begin{aligned}
v_{x,  \lambda}(y) &=\left(\frac{ \lambda}{|y-x|}\right)^{n-2} v\left(x+\frac{ \lambda^{2}(y-x)}{|y-x|^{2}}\right) \\
& \leq C  \lambda^{n-2}|x|^{-\frac{2-\sigma}{1-p}} \\
& \leq C|x|^{n-2-\frac{2-\sigma}{p-1}}
\\
&<\frac{1}{M} \leq \inf _{\partial B_{1}} v\leq v(y),
\end{aligned}\end{equation*}
provided $\epsilon<(\frac{1}{M C})^{(1 /(n-2-\frac{2-\sigma}{p-1}))}$. Since $v_{x,\overline\lambda}(y)=v(y)$ naturally holds on $y\in \partial B_{\overline\lambda}(x)$ and
$-\Delta (v-v_{x,\overline\lambda})\geq 0$
as calculated in \eqref{equ-maximumP}, by maximum principle we have $v_{x,\overline \lambda}(y)<v(y)$ in $B_1\setminus \overline{B_{\overline \lambda}(x)}$ and $\frac{\partial}{\partial\mathcal V}(v-v_{x,\overline\lambda})>0$ on $\partial B_{\overline\lambda}(x)$. By standard argument in moving spheres method as in \cite{15,xiong-2020}, we can move spheres a little further than $\overline\lambda$, contradicting to its definition. This finishes the proof of claim.

Let $r<\epsilon^{2}$ small enough, and $x_{1}, x_{2} \in \partial B_{r}$ such that
\begin{equation*}v\left(x_{1}\right)=\max _{\partial B_{r}} v\quad\text {and}\quad v\left(x_{2}\right)=\min _{\partial B_{r}} v.
\end{equation*}
Let $x_{3}=x_{1}+\frac{\epsilon\left(x_{1}-x_{2}\right)}{4\left|x_{1}-x_{2}\right|}, \lambda=\sqrt{\frac{\epsilon}{4}\left(\left|x_{1}-x_{2}\right|+\frac{\epsilon}{4}\right)}$.
It follows from \eqref{Claim} that
\begin{equation*}v_{x_{3}, \lambda}\left(x_{2}\right) \leq v\left(x_{2}\right).
\end{equation*}
By a direct computation,
\begin{equation*}v_{x_{3}, \lambda}\left(x_{2}\right)=\left(\frac{1}{\frac{4}{\epsilon}\left|x_{1}-x_{2}\right|+1}\right)^{\frac{n-2}{2}} v\left(x_{1}\right) \geq\left(\frac{1}{\frac{8 r}{\epsilon}+1}\right)^{\frac{n-2}{2}} v\left(x_{1}\right),
\end{equation*}
i.e.,
\begin{equation*}\max _{\partial B_{r}} v \leq\left(\frac{8 r}{\epsilon}+1\right)^{\frac{n-2}{2}} \min _{\partial B_{r}} v.
\end{equation*}
By the definition of $\overline{v}(|x|)$ and the Taylor expansion, \eqref{Equ-spherical-average} follows immediately.
\end{proof}

Theorem \ref{Thm-1-Punctured} follows immediately from Propositions \ref{Prop-2.3}, \ref{Prop-2.5} and \ref{prop-2.7}. Theorem \ref{Thm-1} follows from Theorem \ref{Thm-1-Punctured} by transform \eqref{Equ-transform}.

\section{Refined asymptotics}\label{Sec-3}

In this section, we classify the asymptotic behavior of $\bar u(t)$ as $t\rightarrow+\infty$, then the asymptotics of $u(\theta,t)$ follows from \eqref{equ-AsympRadial}. Similar to the classification of asymptotics of positive solutions of \eqref{Equ-averaged-homo} at $+\infty$, we separately discuss the three cases $b=0,a<0$; $b>0,a<0$; $a\geq 0$  in subsections \ref{Sec-3.1}, \ref{Sec-3.2}, \ref{Sec-3.3} respectively.

\subsection{When $b=0,a<0$}\label{Sec-3.1}

By transform \eqref{Equ-transform}, we have $\overline u(t)=r^{\frac{n-2}{2}}\overline v(r)$ for all $t\in(0,+\infty)$. Furthermore, by  elliptic estimates, we have the following estimates on derivatives of error $u-\overline u$.
\begin{lemma}\label{Lem-DerivativeEstimates}
  Let $u(\theta,t)$ be a   positive smooth solution of \eqref{Equ-Cylinder} with $b=0,a<0$ and \eqref{Equ-cond-cylinder}. Then $\overline u(t),\overline u'(t)=O(1)$ as $t\rightarrow+\infty$ and
  \begin{equation}\label{Equ-57}
  \begin{aligned}
\frac{\partial}{\partial t}(u-\overline{u}) &=\overline{u} O\left(e^{-t}\right), \\
\left|\nabla_{\theta}(u-\overline{u})\right| &=\overline{u} O\left(e^{-t}\right).
\end{aligned}\end{equation}
\end{lemma}
\begin{proof}
  By Theorem \ref{Thm-1}, we have $\overline u(t)=O(1)$. When $b=0$, $\bar u(t)$  satisfies \eqref{Equ-averaged} with $b=0$, i.e.,
  \begin{equation*}\overline{u}^{\prime \prime}+a \overline{u}+\overline{u}^{p}=\overline{u}^{p} O\left(e^{-t}\right)\quad\text{in}~(0,+\infty).
  \end{equation*}
  Thus $\overline u''(t)=O(1)$. Multiplying \eqref{Equ-averaged} by $\overline u'(t)$ and integrating by parts, we have  $\overline u'(t)=O(1)$. Furthermore, by \eqref{equ-AsympRadial} we have $|u-\overline u(t)|=O(e^{-t})$ as $t\rightarrow+\infty$.

  For any fixed $t>10$, applying Harnack inequality on $\Omega_{t}=\mathbb S^{n-1} \times[t-2, t+2]$ we have
  \begin{equation}\label{equ-temp-13}
  \max _{\Omega_{t}} u \leq C \min _{\Omega_{t}} u,~
  \max _{[t-1, t+1]} \overline{u} \leq C \overline{u}(t)
  \end{equation}
  for some constant $C$ independent of $t$. By a direct computation, $u-\overline u$ satisfies
  \begin{equation*}
  \partial_{tt}(u-\overline{u})+\Delta_{\mathbb{S}^{n-1}}(u-\overline{u})=F(u(\theta,t),\overline u(t), t)\quad\text{in}~\mathbb S^{n-1}\times(0,+\infty),
  \end{equation*}
  where
  \begin{equation*}F(u,\overline u, t)=-a(u-\overline{u})-u^{p}+\overline{u}^{p}+\overline{u}^{p} O\left(e^{-t}\right)
  \end{equation*}
  is a smooth function. By interior $L^q$-estimates and embeddings in compact Riemannian manifolds as in \cite{Aubin-Riemannian}, for any $0<\gamma<1$, there exists
  $C=C(a,n,\gamma,q)>0$ such that
  \begin{equation*}\begin{aligned}
\|u-\overline{u}\|_{C^{1, \gamma}\left(\Omega_{t_{0}}^{\prime}\right)}
&\leq \|u-\overline{u}\|_{W^{2, q}\left(\Omega_{t_{0}}^{\prime}\right)}\\
& \leq C\left\{\|u-\overline{u}\|_{L^{\infty}\left(\Omega_{t_{0}}\right)}+\|F
(u,\overline u,t)
\|_{L^{q}\left(\Omega_{t_{0}}\right)}\right\} \\
& \leq C\left\{\|u-\overline{u}\|_{L^{\infty}\left(\Omega_{t_{0}}\right)}+\|\overline{u}\|_{L^{\infty}\left(\Omega_{t_{0}}\right)} O\left(e^{-t}\right)\right\},
\end{aligned}\end{equation*}
where $\Omega_{t_{0}}^{\prime}:=\mathbb{S}^{n-1} \times\left[t_{0}-\frac{1}{2}, t_{0}+\frac{1}{2}\right] \subset \subset \Omega_{t_{0}}$. Therefore by the invariance of equation under translation, there exists $C>0$ such that
\begin{equation}\label{equ-temp-14}
\|u-\overline{u}\|_{C^{1, \gamma}\left(\Omega_{t}^{\prime}\right)} \leq C\|\overline{u}\|_{L^{\infty}\left(\Omega_{t}\right)}e^{-t}
\end{equation}
for all $t\geq t_0$.  The result follows immediately from \eqref{equ-temp-13} and \eqref{equ-temp-14}.
\end{proof}

As in the classification result,  all positive smooth solutions of \eqref{Equ-averaged-homo} with $b=0,a<0$ must satisfy $0<\psi\leq\left(-\frac{p+1}{2} a\right)^{\frac{1}{p-1}}$.
Now we prove the following lemma that gives us estimates of initial value problem on the difference.
\begin{lemma}\label{Lem-Esti-Initialvalue}
  For $a<0,$ let $w(t)>0,0<\psi \leq\left(-\frac{p+1}{2} a\right)^{\frac{1}{p-1}}$ be bounded smooth solutions of
  \begin{equation*}w^{\prime \prime}+a w+w^{p}=f(t)\quad\text {in }[0, T), \end{equation*}
  and
  \begin{equation*}\psi^{\prime \prime}+a \psi+\psi^{p}=0\quad\text {in }[0, T),
  \end{equation*}
  where $f(t)$ is smooth and bounded on $[0,T]$. Then there exists a constant $C(T)>0,$ depending only on $n, T, a, p$ and the upper bound of $|w|+\left|w^{\prime}\right|$ such that
  \begin{equation}\label{Equ-temp-estimate}
  \begin{array}{lllll}
  |w(t)-\psi(t)|&+& \left|w^{\prime}(t)-\psi^{\prime}(t)\right| \\
  &\leq&\displaystyle C(T)\left(|w(0)-\psi(0)|+\left|w^{\prime}(0)-\psi^{\prime}(0)\right|+\max _{0 \leq t \leq T}|f(t)|\right).\\
  \end{array}
  \end{equation}
\end{lemma}
\begin{proof}
  Let
  \begin{equation*}\phi(t):=w(t)-\psi(t),\quad A(t):=\left(w^{p}-\psi^{p}\right) /(w-\psi).
  \end{equation*}
  Then $A(t)$ is bounded for all $t\in[0,T]$ relying on the upper bound of $w,\psi$. Let $\mathbf{Y}(t):=\left(y_{1}, y_{2}\right)^{T}:=\left(\phi, \phi^{\prime}\right)^{T}$, then it satisfies
  \begin{equation*}\frac{\mathrm{d} \mathbf{Y}}{\mathrm{d} t}=\mathbf{A}(t) \mathbf{Y}(t)+\mathbf{F}(t),\end{equation*}
  where
  \begin{equation*}\mathbf{A}(t)=\left(\begin{array}{cc}
0 & 1 \\
-a-A(t) & 0
\end{array}\right) \quad \text {and} \quad \mathbf{F}(t)=\left(\begin{array}{c}
0 \\
f(t)
\end{array}\right).
\end{equation*}
Let
\begin{equation*}M:=|w(0)-\psi(0)|+\left|w^{\prime}(0)-\psi^{\prime}(0)\right|+\max _{0 \leq t \leq T}|f(t)|,
$$ and
$$
C_0:=\max_{t\in[0,T]}|-a-A(t)|+1.
\end{equation*}
By Picard-Lindel\"of theorem and the uniqueness of initial value problem, for $k\in\mathbb N, Y^{(0)}:=Y(0)$,
\begin{equation*}\mathbf{Y}^{(k)}(t):=\mathbf{Y}(0)+\int_{0}^{t}\left(\mathbf{A}(s) \mathbf{Y}^{(k-1)}(s)+\mathbf{F}(s)\right) d s
\rightarrow \mathbf Y(t)\quad\text{as}~k\rightarrow+\infty
\end{equation*}
for some solution $\mathbf Y(t)$ satisfying
\begin{equation*}|\mathbf{Y}(t)| \leq \sum_{k=0}^{\infty} \frac{M}{k !}(C_0 t)^{k}=M e^{C_0 t} \leq M e^{C_0 T}.
\end{equation*}
Then \eqref{Equ-temp-estimate} follows immediately by taking $C(T)=e^{C_0T}$.
\end{proof}

Now we give the proof of Theorem \ref{Thm-2} when $b=0,a<0$.
\begin{proof}[Proof of Theorem \ref{Thm-2}.\eqref{Case-1}]
  Multiply \eqref{Equ-Cylinder} by $u_t$ and integrate over $\mathbb S^{n-1}$ to discover
  \begin{equation*}\frac{\mathrm{d}}{\mathrm{d} t} \int_{\mathbb{S}^{n-1}} \frac{1}{2}\left[u_{t}^{2}+a u^{2}+\frac{2}{p+1} u^{p+1}-\left|\nabla_{\theta} u\right|^{2}\right] d \theta=0.
  \end{equation*}
Let
\begin{equation}\label{Equ-def-H}
H(\alpha, \beta):=\alpha^{2}+a \beta^{2}+\frac{2}{p+1} \beta^{p+1},
\end{equation}
and
  \begin{equation*}D(t):=H\left(\overline{u}^{\prime}(t), \overline{u}(t)\right)=\overline{u}^{\prime 2}+a \overline{u}^{2}+\frac{2}{p+1} \overline{u}^{p+1}.
  \end{equation*}
  By Lemma \ref{Lem-DerivativeEstimates}, for any $0<t<s<\infty$, we have
  \begin{equation}\label{equ-temp-15}
  D(t)-D(s)=\left.\left(\overline{u}^{2}+\overline{u}^{\prime 2}\right) O\left(e^{-\cdot}\right)\right|_{s} ^{t}=O(e^{-t})\rightarrow0\quad\text{as}~t\rightarrow+\infty.
  \end{equation}
  Thus $D(t)$ forms a Cauchy sequence and admits a limit $D_{\infty}:=\lim _{t \rightarrow +\infty} D(t)$. Sending $s\rightarrow+\infty$ in \eqref{equ-temp-15}, we have
  \begin{equation}\label{Equ-vanish-D}
  D(t)=D_{\infty}+\left(\overline{u}^{2}(t)+\overline{u}^{\prime 2}(t)\right) O\left(e^{-t}\right).
  \end{equation}
  As in the classification result, we have
  \begin{equation*}-\frac{p-1}{p+1}(-a)^{\frac{p+1}{p-1}} \leq D_{\infty} \leq 0,
  \end{equation*}
  otherwise  $\overline u(t)$ vanishes in finite time.

  When $D_{\infty}=0$, $u$ vanishes at infinity. We start with the proof of that $\overline u$ cannot have critical point for large $t$. By contradiction, we assume there exists $\tilde t_i\rightarrow+\infty$ such that $\overline u'(\tilde t_i)=0$. By \eqref{Equ-57} and \eqref{Equ-vanish-D},
  $$
  a+\dfrac{2}{p+1}\overline u^{p-1}(\tilde t_i)=O(e^{-\tilde t_i}),
  $$
  thus $\overline u(\tilde t_i)\rightarrow (-\frac{p+1}{2}a)^{\frac{1}{p-1}}$ as $i\rightarrow+\infty$. However by \eqref{Equ-averaged}, we have
  $$
  \overline u''(\tilde t_i)+a\overline u(\tilde t_i)+\overline u^p(\tilde t_i)\rightarrow 0,
  $$
  thus $\overline u''(\tilde t_i)>0$ for large $i$, contradicting to the assumption that $\{\tilde t_i\}$ is a sequence of critical points. Thus $\overline u(t)$ is monotone and admitting a limit $\overline u(\infty)$ at infinity. It follows from \eqref{Equ-averaged} that
  $$
  \lim_{t\rightarrow +\infty}\overline u''(t)=0\quad\text{and}\quad
  \lim_{t\rightarrow +\infty}\overline u'(t)=0.
  $$
  By \eqref{Equ-averaged},
  $$
  \lim_{t\rightarrow+\infty}\overline u(t)=0\quad\text{or}~
  \lim_{t\rightarrow+\infty}\overline u(t)=(-a)^{\frac{1}{p-1}}.
  $$
  Since $D_{\infty}=0$, the second case cannot happen. Thus
  \begin{equation*}\lim _{t \rightarrow \infty} \overline{u}(t)=\lim _{t \rightarrow \infty} \overline{u}^{\prime}(t)=0,
  \end{equation*}
  and there exists sufficiently large $t_0\gg 0$ such that
  \begin{equation*}\overline{u}^{\prime}(t)<0\quad \text {for } t>t_{0}.
  \end{equation*}
  By standard ODE analysis as in \cite{Coddington}, using \eqref{Equ-vanish-D}, for any $0<\lambda<\sqrt{-a}$, there exists $T_0>t_0$ relying on $\lambda$ such that
  $$
  \overline u'^2(t)\geq \lambda^2\overline u^2(t)\quad\text{for } t>T_0.
  $$
  Therefore
  \begin{equation*}\overline{u}=O\left(e^{-\lambda t}\right)\quad \text {and}\quad \overline{u}^{\prime}=O\left(e^{-\lambda t}\right).
  \end{equation*}
  By \eqref{Equ-vanish-D} again, for all $t>T_0$,
  \begin{equation*}\frac{-\overline{u}^{\prime}}{\overline{u}} \geq \sqrt{-a}-c e^{-t},
  \end{equation*}
  and hence
  \begin{equation}\label{equ-temp-20}
  \overline{u}(t) \leq C e^{-\sqrt{-a} t}.
  \end{equation}
  The upper bound in \eqref{Equ-case1-D0} follows by
  \eqref{equ-AsympRadial} and \eqref{equ-temp-20}.
  Similarly, for sufficiently large $t$,
  $$
  0<\dfrac{-\overline u'}{\overline u}\leq \sqrt{-a}+ce^{-t}.
  $$
  Thus by a direct calculus,
  there exists $C>0$ such that
  $$
  \overline u(t)\geq Ce^{-\sqrt{-a}t}
  $$
  for sufficiently large $t$. This proves the lower bound in \eqref{Equ-case1-D0}.

  When $D_{\infty}<0$, by \eqref{Equ-vanish-D}, there exists $C_1>0$ such that
  \begin{equation}\label{equ-temp-16}
  \left|H\left(\overline{u}^{\prime}(t), \overline{u}(t)\right)-D_{\infty}\right|
  \leq C_1e^{-t},\quad\forall~t\geq 1.
  \end{equation}
  By Lemma \ref{Lem-DerivativeEstimates} and equation \eqref{Equ-averaged}, $||\overline u(t)||_{C^2(\mathbb S^{n-1}\times(0,+\infty))}$ is bounded. By Arzela-Ascoli theorem, there exists a sequence of $t_i\rightarrow+\infty$ and a positive solution of \eqref{Equ-averaged-homo} with $b=0,a<0$ such that
  \begin{equation}\label{Equ-C1loc-Converge}
  \overline{u}\left(t_{i}+\cdot\right) \rightarrow \psi(\cdot)\quad \text {in } C_{\mathrm{loc}}^{1}(-\infty, \infty)\quad\text {as } i \rightarrow \infty.
  \end{equation}
  By the classification result of \eqref{Equ-averaged-homo} as in Appendix, there exist $0<m \leq M<\left(-\frac{p+1}{2} a\right)^{\frac{1}{p-1}}$ such that
  $$
  0<\inf_{t\in\mathbb R}\psi(t)=m\leq M=\sup_{t\in\mathbb R}\psi(t).
  $$
  If $D_{\infty}=-\frac{p-1}{p+1}(-a)^{\frac{p+1}{p-1}},$ then $\psi \equiv m=(-a)^{\frac{1}{p-1}}$ and thus
  $$
  \lim_{t\rightarrow+\infty}\overline u'(t)=0,\quad\lim_{t\rightarrow+\infty}\overline u(t)=(-a)^{\frac{1}{p-1}}.
  $$
  Since $\nabla H(0, m)=0$ and
  \begin{equation*}\nabla^{2} H(0, m)=\left(\begin{array}{cc}
2 & 0 \\
0 & 2 a(1-p)
\end{array}\right)>0,
\end{equation*}
there exist $A>0$ and $T_0\gg 0$ relying on $A$ such that
\begin{equation*}H\left(\overline{u}^{\prime}, \overline{u}\right) \geq D_{\infty}+A\left(\overline{u}^{\prime 2}(t)+(\overline{u}(t)-m)^{2}\right)
\end{equation*}
for all $t\geq T_0.$ Together with \eqref{equ-temp-16}, we have
$$
|\overline u'(t)|^2+|\overline u(t)-m|^2\leq A^{-1}C_1e^{-t}
$$
for all $t\geq T_0$. Thus we obtain \eqref{Equ-case1-D1} with $\psi\equiv (-a)^{\frac{1}{p-1}},\gamma=\frac{1}{2}$  by \eqref{equ-AsympRadial} and the formula above.

  If $-\frac{p-1}{p+1}(-a)^{\frac{p+1}{p-1}}<D_{\infty}<0$, $\psi$ is a periodic function. There are various similar strategies to prove the convergence rate to periodic functions, see for instance \cite{CoupledSystem,KMPS,Marques-CVPDE}. The proof here is similar to the one by Han-Li-Teixeira \cite{Han-Li-T-Simgak}.  By translating, we assume $\psi(0)=m$ and $\psi^{\prime}(0)=0 .$ Hence
  \begin{equation*}0<\psi^{\prime \prime}(0)=-a m-m^{p}=: l.
  \end{equation*}
  By \eqref{Equ-averaged} and $b=0$, there exists $\epsilon_{1}>0, T_{0}>10$ such that
  \begin{equation}\label{Equ-positive-SecodOrde}
  \overline{u}^{\prime \prime}(t)=-a \overline{u}-\overline{u}^{p}+\overline{u}^{p} O\left(e^{-t}\right) \geq \frac{l}{2}>0
  \end{equation}
  for all $|\overline{u}-m| \leq \epsilon_{1}, t \geq T_{0}$. Since $\left.\partial_{s}\right|_{s=m} H(0, s)=-2 l<0$, for $0<A<2l$, there exists $0<\epsilon_2<\epsilon_1$ such that
  \begin{equation}\label{Equ-temp-18}
  |H(0, s)-H(0, m)| \geq A|s-m|,\quad\forall~|s-m| \leq \epsilon_{2}.
  \end{equation}
  By \eqref{Equ-C1loc-Converge} and Lemma \ref{Lem-Esti-Initialvalue}, we can choose $i_0$ sufficiently large such that $t_{i_{0}} \geq T_{0}+2T$ and
  \begin{equation}\label{Equ-Smallness-pri}
  \left|\overline{u}\left(t_{i_{0}}+t\right)-\psi(t)\right|+\left|\overline{u}^{\prime}\left(t_{i_{0}}+t\right)-\psi^{\prime}(t)\right| \leq \gamma \epsilon_{2},\quad\forall~|t| \leq 2 T,
  \end{equation}
  where $\gamma\in (0,\frac{1}{2})$ to be determined and $T$ is the minimal period of $\psi$.

  We claim that there exists $\left|\delta_{0}\right|<2 l^{-1}\left|\overline{u}^{\prime}\left(t_{i_{0}}\right)\right|$ such that
  \begin{equation}\label{Equ-claim}
  \overline{u}^{\prime}\left(t_{i_{0}}+\delta_{0}\right)=0.
  \end{equation}
  Let $\Lambda:=\sup _{\mathbb{R}}\left\{\left|\psi^{\prime \prime}\right|,\left|\psi^{\prime}\right|\right\}$, then
  \begin{equation*}|\psi(t)-m|+\left|\psi^{\prime}(t)\right| \leq 2 \Lambda|t|<\frac{\epsilon_{2}}{2},\quad\forall~|t|<\frac{\epsilon_2}{4\Lambda}.
  \end{equation*}
  Together with \eqref{Equ-Smallness-pri}, we have
  \begin{equation}\label{equ-temp-17}
  \left|\overline{u}\left(t_{i_{0}}+t\right)-m\right|+\left|\overline{u}^{\prime}\left(t_{t_{0}}+t\right)\right|<\epsilon_{2},\quad\forall~|t|<\frac{\epsilon_2}{4\Lambda}.
  \end{equation}
  By \eqref{Equ-positive-SecodOrde}, $\overline{u}^{\prime \prime}\left(t_{i_{0}}+t\right) \geq \frac{l}{2}$ for all $|t|\leq\frac{\epsilon_2}{4\Lambda}$. With \eqref{Equ-Smallness-pri}, there exists $\delta_0$ such that $\overline u'(t_i+\delta_0)=0$ with estimate
  \begin{equation*}\left|\delta_{0}\right| \leq \frac{2}{l}\left|\overline{u}^{\prime}\left(t_{i_{0}}\right)\right| \leq \frac{2}{l} \gamma \epsilon_{2}<\frac{\epsilon_{2}}{4 \Lambda},
  \end{equation*}
  where $\gamma<\frac{l}{8\Lambda}$ is fixed relying only on $l,\Lambda$. This proves the claim \eqref{Equ-claim}. Furthermore, with \eqref{equ-temp-17} we have $\left|\overline{u}\left(t_{i_{0}}+\delta_{0}\right)-m\right|<\epsilon_{2}$. Therefore by \eqref{Equ-temp-18} and \eqref{equ-temp-16},
  $$
  \begin{array}{lllll}
    A\left|\overline{u}\left(t_{i_{0}}+\delta_{0}\right)-m\right|
    &\leq & \left|H\left(0, \overline{u}\left(t_{i_{0}}+\delta_{0}\right)\right)-H(0, m)\right|\\
    &=& \left|H\left(\overline{u}^{\prime}\left(t_{i_{0}}+\delta_{0}\right), \overline{u}\left(t_{i_{0}}+\delta_{0}\right)\right)-H(0, m)\right|\\
    &=& \left|H\left(\overline{u}^{\prime}\left(t_{i_{0}}+\delta_{0}\right), \overline{u}\left(t_{i_{0}}+\delta_{0}\right)\right)-D_{\infty}\right|\\
    &\leq &C_{0} e^{-\left(t_{i_{0}}+\delta_{0}\right)},
  \end{array}
  $$
  i.e.,
  \begin{equation*}\left|\overline{u}\left(t_{i_{0}}+\delta_{0}\right)-m\right| \leq C_{0} A^{-1} e^{-\left(t_{i_{0}}+\delta_{0}\right)}.\end{equation*}
  Let $\tau_{0}=t_{i_{0}}+\delta_{0}$, by Lemma \ref{Lem-Esti-Initialvalue},
  \begin{equation*}\begin{array}{ll}
\left|\overline{u}\left(\tau_{0}+t\right)-\psi(t)\right|+\left|\overline{u}^{\prime}\left(\tau_{0}+t\right)-\psi^{\prime}(t)\right| \\
\leq C(T)\left(\left|\overline{u}\left(\tau_{0}\right)-m\right|+C_{2} e^{-\tau_{0}}\right) & \\
\leq C(T)\left(C_{0} A^{-1}+C_{2}\right) e^{-\tau_{0}}& \text { for } 0 \leq t \leq 2 T.
\end{array}\end{equation*}
Repeating the arguments above, choosing large $T_0$ such that $C(T)\left(C_{0} A^{-1}+C_{2}\right) e^{-t}<\gamma \epsilon_{2}$ for all $t\geq T_0$, we obtain $\delta_{1}$ such that $\tau_{1}=\tau_{0}+T+\delta_{1}$ satisfies
\begin{equation*}\begin{array}{l}
\overline{u}^{\prime}\left(\tau_{1}\right)=0; \\
\left|\delta_{1}\right| \leq 2 l^{-1}\left|\overline{u}^{\prime}\left(\tau_{0}+T\right)\right| \leq 2 l^{-1} C(T)\left(C_{0} A^{-1}+C_{2}\right) e^{-\tau_{0}}; \\
\left|\overline{u}\left(\tau_{1}\right)-m\right| \leq C_{1} A^{-1} e^{-\tau_{1}}.
\end{array}\end{equation*}
We can now inductively obtain $\tau_{j}=\tau_{j-1}+T+\delta_{j}$ such that
\begin{equation*}\begin{array}{l}
\overline{u}^{\prime}\left(\tau_{j}\right)=0; \\
\left|\delta_{j}\right| \leq 2 l^{-1}\left|\overline{u}^{\prime}\left(\tau_{j-1}+T\right)\right| \leq 2 l^{-1} C(T)\left(C_{0} A^{-1}+C_{2}\right) e^{-\tau_{j-1}}; \\
\left|\overline{u}\left(\tau_{j}\right)-m\right| \leq C_{1} A^{-1} e^{-\tau_{j}},\quad \text{for } 0 \leq t \leq 2 T,
\end{array}\end{equation*}
and
\begin{equation}\label{equ-423}
\begin{array}{llll}
&\left|\overline{u}\left(\tau_{j}+t\right)-\psi(t)\right|+\left|\overline{u}^{\prime}
\left(\tau_{j}+t\right)-\psi^{\prime}(t)\right|\\
\leq & C(T)\left(C_{0} A^{-1}+C_{2}\right) e^{-\tau_{j}}
\end{array}
\end{equation}
for all $0\leq t\leq 2T$. Set $s_{j}=\tau_{j}-jT$,  then $s_{j}=s_{j-1}+\delta_{j}$ and
\begin{equation*}s_{\infty}:=\lim _{j \rightarrow \infty} s_{j}=\sum_{j=1}^{\infty} \delta_{j}<\infty.
\end{equation*}
Applying \eqref{equ-423} with $\tau=\tau_{j}+t,$ for $\tau_{j} \leq \tau \leq \tau_{j+1}$,
\begin{equation*}\begin{aligned}
&\left|\overline{u}(\tau)-\psi\left(\tau-s_{j}\right)\right|+\left| \overline{u}^{\prime}(\tau)-\psi^{\prime}\left(\tau-s_{j}\right)\right| \\
=&\left|\overline{u}(\tau)-\psi\left(\tau-\tau_{j}\right)\right|+\left|\overline{u}^{\prime}(\tau)-\psi^{\prime}\left(\tau-\tau_{j}\right)\right| \\
\leq & C(T)\left(C_{0} A^{-1}+C_{2}\right) e^{-\tau_{j}} \leq C e^{-\tau_{j}}.
\end{aligned}\end{equation*}
Thus
$$
\begin{array}{lllll}
  & \left|\overline{u}(\tau)-\psi\left(\tau-s_{\infty}\right)\right|+\left|\overline{u}^{\prime}(\tau)-\psi^{\prime}\left(\tau-s_{\infty}\right)\right|\\
\leq & \left|\overline{u}(\tau)-\psi\left(\tau-s_{j}\right)\right|+\left|\overline{u}^{\prime}(\tau)-\psi^{\prime}\left(\tau-s_{j}\right)\right|
\\
&+\left|\psi\left(\tau-s_{j}\right)-\psi(\tau-s_{ \infty})\right|
+\left|\psi^{\prime}\left(\tau-s_{j}\right)-\psi^{\prime}\left(\tau-s_{\infty}\right)\right|\\
\leq & C e^{-\tau_{j}}+\Lambda\left|s_{j}-s_{\infty}\right|\\
\leq & C e^{-\tau_{j}}+\Lambda \displaystyle \sum_{k=j+1}^{\infty}\left|\delta_{k}\right|\\
\leq &C e^{-\tau_{j}} \leq C e^{-\tau}.
\end{array}
$$
Thus we obtain \eqref{Equ-case1-D1} with $\gamma=1$ by \eqref{equ-AsympRadial} and the formula above. This completes the proof of case \eqref{Case-1} in Theorem \ref{Thm-2}.
\end{proof}

\subsection{When $b>0,a<0$}\label{Sec-3.2}

As proved in the Appendix, there are only two global non-negative solutions $\psi\equiv (-a)^{\frac{1}{p-1}}$ and $\psi\equiv 0$ of \eqref{Equ-averaged-homo}. Here in this subsection, we prove that for local positive solutions, $\overline u(t)$ also converge to either of these two quantities.

\begin{lemma}\label{Lem-3.3}
  Let $u$ be a positive solution of \eqref{Equ-Cylinder} and $a,b,p$ verifies \eqref{Equ-cond-cylinder} with $b>0,a<0$. Then
  \begin{equation}\label{Equ-limit-subc}
  \lim_{t\rightarrow+\infty}\overline u''(t)=\lim_{t\rightarrow+\infty}\overline u'(t)=0\quad\text{and}\quad\lim_{t\rightarrow+\infty}\overline u(t)=0~\text{or}~(-a)^{\frac{1}{p-1}}.
  \end{equation}
\end{lemma}
\begin{proof} The proof is similar to the one of Theorem 1.3 in \cite{CGS}.
  By Theorem \ref{Thm-1}, $\overline u(t)=O(1)$ and $\overline u'(t)=O(1)$ as $t\rightarrow+\infty$. By \eqref{Equ-averaged},  $\overline u''(t)=O(1)$. Multiply \eqref{Equ-averaged} by $\overline u'(t)$ and integral over $0<s<t$ to obtain
  \begin{equation*}\left.\frac{1}{2} \overline{u}^{\prime 2}\right|_{s} ^{t}+\int_{s}^{t} b \overline{u}^{2} d \tau+\left.\frac{1}{2}a \overline{u}^{2}\right|_{s} ^{t}+\left.\frac{\overline{u}^{p+1}}{p+1}\right|_{s} ^{t}+O\left(e^{-s}\right)=0.
  \end{equation*}
  Since $\overline u',\overline u$ are bounded, sending $t\rightarrow+\infty$ and we have $\int_{s}^{\infty} \overline{u}^{\prime 2}(\tau) d \tau<\infty$. Now we claim that
  \begin{equation}\label{equ-limit-1der}
    \lim_{t\rightarrow+\infty}\overline u'(t)=0.
  \end{equation}
  By contradiction, we assume that \eqref{equ-limit-1der} fails, then there exists $\epsilon>0$ and a sequence of $t_j\rightarrow+\infty$ such that $|\overline u'(t_j)|\geq 2\epsilon$. Let $L:=\sup_{t\in\mathbb R^+}|\overline u''(t)|$, then $|\overline u'(t)|\geq \epsilon$ for all $|t-t_j|\leq \frac{\epsilon}{L}$. By choosing subsequence, we may assume $t_j>t_{j-1}+\frac{2\epsilon}{L}$. Thus for any finite $N$,
  $$
  \int_{s}^{\infty} \overline{u}'^2(\tau) d \tau\geq
  \sum_{j=1}^N\int_{t_j}^{t_{j+1}}\overline u'^2(\tau)d\tau\geq \dfrac{2\epsilon^3}{L}N\rightarrow +\infty
  $$
  as $N\rightarrow+\infty$, contradicting to the finite integral proved above. Using similar technique, multiplying \eqref{Equ-averaged} by $\overline u''(t)$ and we can prove $\lim_{t\rightarrow+\infty}\overline u''(t)=0$.

  Thus sending $t\rightarrow+\infty$ in \eqref{Equ-averaged}, we have
  $$
  \lim_{t\rightarrow+\infty}a\overline u(t)+\overline u^p(t)=0.
  $$
  The result follows immediately since $u$ is a smooth solution.
\end{proof}
Based on \eqref{Equ-limit-subc}, it remains to prove the convergence speed of $\overline u(t)$, then case \eqref{Case-2} in Theorem \ref{Thm-2} follows immediately. We start with the easier case where $\overline u(t)\rightarrow 0$ at infinity.

\begin{proof}[Proof of Theorem \ref{Thm-2}.\eqref{Case-2}.\eqref{equ-case2-D0}]
  By standard ODE technique as in \cite{Coddington}, the decay rate of $\overline u(t)$ is controlled by the negative root of the characteristic equation. Let $\lambda_1,\lambda_2$ be the roots of
  $$
  \lambda^2+b\lambda+a=0,
  $$
  i.e.,
  $$
  \lambda_{1}:=\dfrac{-b-\sqrt{b^2-4a}}{2}<0<\lambda_2:=\dfrac{-b+\sqrt{b^2-4a}}{2}.
  $$
  In fact, by \eqref{Equ-averaged} and $\overline u(t)\rightarrow 0$, for any $\epsilon>0$ sufficiently small to be determined, there exists $T_0\gg 1$ such that
  $$
  \overline u''+b\overline u'+(a+\epsilon)\overline u>0,\quad\forall~ t\geq T_0.
  $$
  Since $a<0$ and $\overline u(t)>0$, we may choose $\epsilon$ small such that $a+\epsilon<0$, then
  we have maximum principle as the following. If $\varphi\in C^2([T_0,+\infty))$ is a classical solution of
  $$
  \left\{
  \begin{array}{llll}
    \varphi''+b\varphi'+(a+\epsilon)\varphi\leq 0,& \text{in }(T_0,+\infty),\\
    \varphi(T_0)\geq 0, & \text{at }t=T_0,\\
    \varphi(t)\rightarrow 0, & \text{as }t\rightarrow+\infty,\\
  \end{array}
  \right.
  $$
  then $\varphi\geq 0$.
  Let
  $$
  \lambda_{i,\epsilon}:=\dfrac{-b\mp\sqrt{b^2-4(a+\epsilon)}}{2},\quad i=1,2.
  $$
  Then $e^{\lambda_{1,\epsilon}t}$ is a  positive solutions of homogeneous equation $\psi_{\epsilon}''(t)+b\psi_{\epsilon}'
  +(a+\epsilon)\psi_{\epsilon}=0
  $. Thus by maximum principle as above, we have
  \begin{equation}\label{equ-estimate-rough-2}
  0\leq \overline u(t)\leq Ce^{\lambda_{1,\epsilon}t}
  \end{equation}
  for some $C>0$, for all $t\geq T_0$.

  Furthermore, by \eqref{Equ-averaged} and
  Wronskian formula, there exist  constants $C_1,C_2$ such that for all $t\geq T_0$,
  \begin{equation}\label{Equ-wronski-3}
  \begin{array}{llll}
  \overline u(t)&=& C_1e^{\lambda_1t}+C_2e^{\lambda_2t}\\
  &&\displaystyle-e^{\lambda_1t}\int_{T_0}^t
  \dfrac{e^{\lambda_2\tau}}{W(\tau)}F(\tau)d\tau
  +e^{\lambda_2t}\int_{+\infty}^t
  \dfrac{e^{\lambda_1\tau}}{W(\tau)}F(\tau)d\tau,\\
  \end{array}
  \end{equation}
  where $W(\tau)$ is the Wronskian determinant $$
  W(\tau)=(\lambda_2-\lambda_1)e^{(\lambda_1+\lambda_2)\tau}=\sqrt{b^2-4a}e^{(\lambda_1+\lambda_2)\tau},
  $$
  and
  $$
  F(\tau)=\overline u^p(\tau)\cdot (-1+O(e^{-\tau}))\quad\text{as}~\tau\rightarrow 0.
  $$
  Since $\lambda_2>0$ and $\overline u(t)\rightarrow 0$ at infinity, we have $C_2=0$ by triangle inequality.

  By fixing $\epsilon$ sufficiently small such that $p\lambda_{1,\epsilon}<\lambda_1<0$, we put \eqref{equ-estimate-rough-2} into Wronskian formula \eqref{Equ-wronski-3} to obtain
  $$
  \overline u(t)\leq Ce^{\lambda_1t}+
  Ce^{\lambda_1t}\int_{T_0}^{t}e^{-\lambda_1\tau}\cdot e^{p\lambda_{1,\epsilon}\tau}d\tau
  +Ce^{\lambda_2t}\int_t^{+\infty}e^{-\lambda_2\tau}\cdot e^{p\lambda_{1,\epsilon}\tau}d\tau,
  $$
  i.e.,
  $\overline u(t)\leq Ce^{\lambda_1t}$
  and the upper bound in \eqref{equ-case2-D0} follows immediately from \eqref{equ-AsympRadial}. Similarly, we  choose sufficiently small $\epsilon'<0$ such that $a+\epsilon'<0$ and $p\lambda_1<\lambda_{1,\epsilon'}$,
  where
  $$
  \lambda_{1,\epsilon'}:=\dfrac{-b-\sqrt{b^2-4(a+\epsilon')}}{2}.
  $$
   By maximum principle as above, we have
  \begin{equation}\label{Equ-temp-temp}
  \frac{1}{C}e^{\lambda_{1,\epsilon'}t}\leq \overline u(t)\leq Ce^{\lambda_1t}
  \end{equation}
  for some $C>0$ and large $t$.
  By \eqref{Equ-wronski-3} and the upper estimate in \eqref{Equ-temp-temp},
  there exists constant $C_3$ such that
  $$
  \bar u(t)=C_3e^{\lambda_1t}+O(e^{p\lambda_1t}).
  $$
  By the lower bound in \eqref{Equ-temp-temp} and triangle inequality, we have $C_3>0$.
  Thus $\bar u(t)\geq Ce^{\lambda_1t}$ for some $C>0$ and sufficiently large $t$. This proves the lower bound in \eqref{equ-case2-D0}.
\end{proof}

Now we prove the convergence speed of $\overline u(t)\rightarrow (-a)^{\frac{1}{p-1}}$ at infinity.

\begin{proof}[Proof of Theorem \ref{Thm-2}.\eqref{Case-2}.\eqref{equ-case2-D1}]
  Let
  \begin{equation*}\xi(t):=\overline{u}(t)-c_{0},
  ~\text{where}~c_{0}:=(-a)^{\frac{1}{p-1}}>0.
  \end{equation*}
  By a direct computation and Lemma \ref{Lem-3.3}, \eqref{Equ-averaged} implies
  \begin{equation*}\left\{\begin{array}{ll}
\xi^{\prime \prime}(t)+b \xi^{\prime}(t)+(p-1) c_{0}^{p-1} \xi(t)=F(\xi, t) & \text { in } t \in(0,+\infty), \\
\xi(t) \rightarrow 0 & \text { as } t \rightarrow \infty, \\
\xi^{\prime}(t) \rightarrow 0 & \text { as } t \rightarrow \infty,
\end{array}\right.\end{equation*}
where
\begin{equation}\label{equ-temp-Wron}
|F(\xi, t)|\leq
C_1|\xi(t)|^2+C_2|\xi(t)|e^{-t}+C_3c_0^pe^{-t}
\end{equation}
for some constants $C_1,C_2,C_3>0$ for all $t>1$.
First we analyze its homogeneous part, i.e.,
\begin{equation}\label{equ-temp-homo}
\phi^{\prime \prime}(t)+b \phi^{\prime}(t)+c_{0}^{p-1}(p-1) \phi(t)=0,~t\in\mathbb R.
\end{equation}
Let $\lambda_{1},\lambda_2$ be the   roots of the characteristic equation $\lambda^{2}+b \lambda+(p-1) c_{0}^{p-1}=0$, i.e.,
\begin{equation*}\lambda_{1}:=\frac{-b-\sqrt{b^{2}-4(p-1) c_{0}^{p-1}}}{2},\quad
\lambda_2:=
\frac{-b+\sqrt{b^{2}-4(p-1) c_{0}^{p-1}}}{2}
.\end{equation*}
Thus \eqref{equ-temp-homo} admits two linearly independent solutions $\phi_{i},~i=1,2$, as the following
$$
\phi_1(t):=e^{\lambda_1t},~\phi_2t=\left\{
\begin{array}{llll}
  e^{\lambda_2t}, & \text{if }b^{2} \neq 4(p-1) c_{0}^{p-1},\\
  te^{-\frac{b}{2}t}, & \text{if }b^{2}= 4(p-1) c_{0}^{p-1}.\\
\end{array}
\right.
$$
Let $\mu_i:=\text{Re}(\lambda_i)$. Since $b>0,c_0>0$, we have $\mu_{1} \leq \mu_{2}<0$. Let
$$
\alpha_0:=\min\{-\mu_2,1\}>0.
$$

When $b^2\not= 4(p-1)c_0^{p-1}$,
by Wronskian formula, for fixed $T_0\gg 0$, there exist two constants $C_4,C_5$ such that
\begin{equation}\label{equ-Wronski}
\xi(t)=C_4e^{\lambda_1t}+C_5e^{\lambda_2t}-e^{\lambda_1t}\int_{T_0}^t
\dfrac{e^{\lambda_2\tau}}{W(\tau)}F(\xi(\tau),\tau)d\tau
+e^{\lambda_2t}\int_{T_0}^t
\dfrac{e^{\lambda_1\tau}}{W(\tau)}F(\xi(\tau),\tau)d\tau,
\end{equation}
where $W(\tau)$ is the Wronskian determinant
$$
W(\tau)=(\lambda_2-\lambda_1)e^{(\lambda_1+\lambda_2)\tau}=\sqrt{b^2-4(p-1)c_0^{p-1}}
e^{(\lambda_1+\lambda_2)\tau}.
$$
Then there exists $K>0$ such that $|W(\tau)e^{-(\lambda_1+\lambda_2)\tau}|>K$.
For any sufficiently small $\epsilon>0$, there exists $T_0\gg 1$ such that for all $t>T_0$,
$$
\dfrac{C_1}{K}|\xi(t)|<\dfrac{\epsilon}{4},~C_2e^{-t}<\dfrac{\epsilon}{4}.
$$
Thus
$$
|\xi(t)|\leq Ce^{\mu_1t}+Ce^{\mu_2t}+\dfrac{\epsilon}{4}e^{\mu_1t}
\int_{T_0}^te^{-\mu_1\tau}|\xi(\tau)|d\tau
+\dfrac{\epsilon}{4}e^{\mu_2t}
\int_{T_0}^te^{-\mu_2\tau}|\xi(\tau)|d\tau
+Ce^{-t}.
$$
Since $\mu_1\leq\mu_2$,  we have
$$
\begin{array}{lllllll}
e^{-\mu_2t}|\xi(t)| & \leq &
\displaystyle Ce^{(\mu_1-\mu_2)t}+C+Ce^{-(\mu_2+1)t}
+\dfrac{\epsilon}{4}e^{(\mu_1-\mu_2)t}\int_{T_0}^te^{-\mu_1\tau}|\xi(\tau)|d\tau
+\dfrac{\epsilon}{4}\int_{T_0}^te^{-\mu_2\tau}|\xi(\tau)|d\tau\\
&\leq & \displaystyle Ce^{(\mu_1-\mu_2)t}+C+Ce^{-(\mu_2+1)t}
+\dfrac{\epsilon}{4}\int_{T_0}^te^{(\mu_1-\mu_2)\tau}e^{-\mu_1\tau}|\xi(\tau)|d\tau
+\dfrac{\epsilon}{4}\int_{T_0}^te^{-\mu_2\tau}|\xi(\tau)|d\tau\\
&= &\displaystyle Ce^{(\mu_1-\mu_2)t}+C+Ce^{-(\mu_2+1)t}
+\dfrac{\epsilon}{2}\int_{T_0}^te^{-\mu_2\tau}|\xi(\tau)|d\tau.\\
\end{array}
$$
By the Gr\"onwall's inequality,
$$
 e^{-\mu_2t}|\xi(t)|\leq Ce^{(\mu_1-\mu_2)t}+C +Ce^{-(\mu_2+1)t}
 +\int_{T_0}^t\left(
 Ce^{(\mu_1-\mu_2)s}+C+Ce^{-(\mu_2+1)s}
 \right) \frac{\epsilon}{2}\exp\left(\int_s^t\frac{\epsilon}{2}d \tau\right)d s,
$$
thus
\begin{equation}\label{equ-estimate-rough}
|\xi(t)|\leq Ce^{-(\alpha_0-\frac{\epsilon}{2})t}=:Ce^{-\alpha t}
\end{equation}
for some constant $C>0$ for sufficiently large $t$.
Put \eqref{equ-estimate-rough} into \eqref{equ-temp-Wron} and \eqref{equ-Wronski} to obtain
$$
|F(\xi(t),t)|\leq Ce^{-2\alpha t}
+Ce^{-(\alpha+1)t}+Ce^{-t}\leq Ce^{-\min\{2\alpha,1\}t}
$$
 for all $t\geq T_0$ and
\begin{equation*}
|\xi(t)| \leq C e^{-\alpha_{0} t}+C e^{\mu_{1} t} \int_{T_0}^{t} e^{\left(-\mu_{1}-\min \{2 \alpha, 1\}\right) s} \mathrm{d} s+C e^{\mu_{2} t} \int_{T_0}^{t} e^{\left(-\mu_{2}-\min \{2 \alpha, 1\}\right) s} \mathrm{d} s.
\end{equation*}
Since $\epsilon$ can be taken arbitrarily small, we have $2\alpha>\alpha_0$ and thus
\begin{equation*}|\xi(t)| \leq\left\{\begin{array}{ll}
C t e^{-\alpha_{0} t}, & \text { if }\mu_{2}=-1, \\
C e^{-\alpha_{0} t}, & \text { otherwise. }
\end{array}\right.\end{equation*}

When $b^{2}=4(p-1) c_{0}^{p-1}$, then the Wronskian formula becomes
\begin{equation}\label{equ-Wronski-2}
\xi(t)=C_{4} e^{-\frac{b}{2} t}+C_{5} t \cdot e^{-\frac{b}{2} t}+e^{-\frac{b}{2} t} \int_{T_0}^{t} e^{\frac{b}{2} \tau} F(\xi(\tau),\tau) \mathrm{d} \tau-e^{-\frac{b}{2} t} \int_{T_0}^{t} \tau e^{\frac{b}{2} \tau} F(\xi(\tau),\tau) \mathrm{d} \tau.
\end{equation}
Similarly, we still have \eqref{equ-estimate-rough} by the Gr\"onwall's inequality since $t e^{-\frac{\epsilon}{4}t}\rightarrow 0$ as $t\rightarrow+\infty$. Put \eqref{equ-estimate-rough} into \eqref{equ-Wronski-2} again, we have
\begin{equation*}|\xi(t)| \leq C t \cdot e^{-\frac{b}{2} t}+C e^{-\frac{b}{2} t} \int_{T_0}^{t} \tau e^{\left(\frac{b}{2}-\min \{2 \alpha, 1\}\right) \tau} \mathrm{d} \tau.
\end{equation*}
Since $\epsilon$ can be taken arbitrarily small, we have
$$
|\xi(t)|\leq \left\{
\begin{array}{llll}
  Ct^2e^{-\alpha_0t}, & \text{if }b=2,\\
  Cte^{-\alpha_0t}, & \text{if }b\not=2.
\end{array}
\right.
$$
Combining these two cases \eqref{equ-case2-D1} follows immediately by the definition of $\xi(t)$ and \eqref{equ-AsympRadial}.
\end{proof}

\subsection{When $a\geq 0$}\label{Sec-3.3}

From Appendix, when $a\geq 0$, there is only one global non-negative solution $\psi\equiv 0$
of \eqref{Equ-averaged-homo}. In this subsection, we prove that for both $b=0$ and $b>0$, $\overline u(t)\rightarrow 0$ as $t\rightarrow+\infty$. The convergence speed follows similar to the proof of Theorem \ref{Thm-2}.\eqref{Case-2}.\eqref{equ-case2-D0} case.

When $b=0,a\geq 0$, we still have \eqref{Equ-vanish-D} as in $a<0$ case. As in the classification result on asymptotics of \eqref{Equ-averaged-homo}, we have $D_{\infty}=0$, otherwise  $\overline u(t)$ vanishes in finite time. Thus
$$
\lim_{t\rightarrow+\infty}\overline u'^2(t)+a\overline u^2(t)+\dfrac{2}{p+1}\overline u^{p+1}(t)=0,
$$
and hence $\overline u(t)\rightarrow 0$ as $t\rightarrow+\infty$.
When $b>0,a\geq 0$, we also have $\overline u(t)\rightarrow 0$ as $t\rightarrow+\infty$ by the proof of Lemma \ref{Lem-3.3}.

To obtain an estimate of decay speed, we analyze \eqref{Equ-averaged} by the roots of the characteristic equation of the homogeneous part, i.e.,
\begin{equation}\label{equ-homo}
\phi''(t)+b\phi'(t)+a\phi(t)=0,~-\infty<t<\infty.
\end{equation}
Let $\lambda_1,\lambda_2$ be the two roots of
$
\lambda^2+b\lambda+a=0
$, i.e.,
$$
\lambda_1:=\dfrac{-b-\sqrt{b^2-4a}}{2},\quad\lambda_2:=\dfrac{-b+\sqrt{b^2-4a}}{2}.
$$

When $b=0$, $\text{Re}\lambda_1=\text{Re}\lambda_2=0$.
When $b>0, a=0$, $\text{Re}\lambda_2=0$.
As stated in \cite{Coddington}, it remains a difficult question to estimate the vanishing speed when the eigenvalue have zero real part. Thus we merely obtain $\overline u(t)=o(1)$ in these cases. Especially for $a=0, b=n-2$ case, Aviles \cite{4} proved an optimal estimate $u(\theta,t)=t^{-\frac{1}{p-1}}O(1)$ as $t\rightarrow+\infty$.

When $b>0, b^2-4a>0$ and $a>0$, we have $\lambda_1<\lambda_2<0$. Thus by similar argument as in Theorem \ref{Thm-2}.\eqref{Case-2}.\eqref{equ-case2-D1}, $|\overline u(t)|=O(e^{\lambda_2t})$. When $b>0, b^2-4a=0$, we have $|\overline u(t)|=O(te^{\lambda_2t})$. When $b>0, b^2-4a<0$, we have $|\overline u(t)|=O(e^{-\frac{b}{2}t})$.  By transform \eqref{Equ-transform}, maximum principle as Lemma 2.1 of \cite{CCLin-Duke95} and $v(x)>0$ in $B_1\setminus\{0\}$, we have $u(\theta,t)>Ce^{-\frac{n+b-2}{2}t}$ for some $C>0$ and large $t$.

Combining all these case, \eqref{equ-case3-D} follows from \eqref{equ-AsympRadial} and the estimates above.

\section{Removability of singularity}\label{Sec-4}

In this section, we discuss further when $u(\theta,t)\rightarrow 0$ as in \eqref{Equ-case1-D0}, \eqref{equ-case2-D0} or \eqref{equ-case3-D}, whether the corresponding positive solution $v(x)$ of \eqref{Equ-Punctured} has a removable singularity at origin.
First, we consider $c=0$, in which case by standard bootstrap argument if $v\in C^0(B_1)$ then $v\in C^{\infty}(B_1)$,  see for instance \cite{CGS,CoupledSystem,GT,Marques-CVPDE,xiong-2020}.
\begin{proposition}
Let $v$ be a positive smooth solution of \eqref{Equ-Punctured} in $B_1\setminus\{0\}$ with
\begin{equation}\label{equ-cond-classical}
c=0,\quad\sigma\in[0,2)\quad\text{and}\quad
\dfrac{n-\sigma}{n-2}<p\leq\dfrac{n+2-2\sigma}{n-2}.
\end{equation}
Then either $v$ can be extended smoothly to $B_1$ or $v(x)\approx |x|^{-\frac{2-\sigma}{p-1}}$ as $x\rightarrow 0$.
\end{proposition}
\begin{proof}
  By transform \eqref{Equ-transform}, \eqref{equ-cond-classical} gives us
  $$
  b^2-4a=(n-2)^2\quad\text{and}\quad0\leq b<n-2.
  $$
  Thus $a<0$. By Theorem \ref{Thm-2} and transform \eqref{Equ-transform}, when \eqref{Equ-case1-D1}  or \eqref{equ-case2-D1} happens, $v\approx |x|^{-\frac{2-\sigma}{p-2}}$ as $x\rightarrow 0$ and $v\not\in H^1_{loc}(B_1)$.

  When \eqref{Equ-case1-D0} or $\eqref{equ-case2-D0}$ happens, we have
  $$
  \dfrac{1}{C}\leq v(x)=|x|^{-\frac{2-\sigma}{p-1}}u(\theta,-\ln |x|)\leq C
  $$
  in a neighbourhood of origin. Thus in this case, $v$ can be extend to $B_1$ smoothly.
\end{proof}

\begin{proposition}
  Let $v$ be a positive smooth solution of \eqref{Equ-Punctured} in $B_1\setminus\{0\}$ with
  $$
0<c<\frac{(n-2)^2}{4},\quad \sigma\in[0,2)\quad\text{and}\quad
\frac{n+2+\sqrt{(n-2)^2-4c}-2\sigma}{n-2+\sqrt{(n-2)^2-4c}}<p\leq \frac{n+2-2\sigma}{n-2}.
$$
Then either  $v\in H^1_{loc}(B_1)$ but not bounded near origin or $v\approx |x|^{-\frac{2-\sigma}{p-1}}$ as $x\rightarrow 0$.
\end{proposition}
\begin{proof}
  By transform \eqref{Equ-transform}, we have \eqref{Equ-cond-cylinder} with $a<0$ and
  $$
  0<b^2-4a<(n-2)^2.
  $$
  By Theorem \ref{Thm-2}, either
  $$
  \liminf_{t\rightarrow+\infty}u(\theta,t)>0
  $$
  or \eqref{Equ-case1-D0} (when $p=\frac{n+2-2\sigma}{n-2}$) or \eqref{equ-case2-D1} (when $p<\frac{n+2-2\sigma}{n-2}$) holds.

  For the first case, we have $v\approx |x|^{-\frac{2-\sigma}{p-1}}$ as $x\rightarrow 0$. For the later two cases,
  $$
  \dfrac{1}{C}|x|^{\frac{\sqrt{b^2-4a}}{2}-\frac{n-2}{2}} \leq
  v\leq C|x|^{\frac{\sqrt{b^2-4a}}{2}-\frac{n-2}{2}},\quad\forall~x\in B_{\frac{1}{2}}\setminus\{0\},$$
  for some $C>0$.
Thus $v$ is not bounded near origin and $v\in L^{p_0}_{loc}(B_1)$ for some $p_0>\frac{2n}{n-2}$. By the H\"older inequality and $L^p$-estimates, we have
$\frac{c}{|x|^2}v+\frac{v^p}{|x|^{\sigma}}\in L^{q}_{loc}(B_1)$ for some $q>\frac{2n}{n+2}$ and hence $v\in H^1_{loc}(B_1)$.
\end{proof}

\begin{remark}
  Although when \eqref{Equ-case1-D0} or \eqref{equ-case2-D0} holds, $0$ is an $H^1$-removable singularity of $v$, but it is not a removable singularity. From the proof of Theorem \ref{Thm-2}, the estimates in \eqref{Equ-case1-D0} and \eqref{equ-case2-D0} are optimal.
\end{remark}

When \eqref{equ-case3-D} holds with $a,b>0$, from the proof in Theorem \ref{Thm-1}.\eqref{Case-3}, we have
$$
v(x)\leq C\left\{
\begin{array}{lll}
  |x|^{\frac{b-\sqrt{b^2-4a}}{2}-\frac{2-\sigma}{p-1}}, & \text{if }b^2-4a>0,\\
  |x|^{\frac{b}{2}-\frac{2-\sigma}{p-1}}\ln |x|, & \text{if }b^2-4a=0,\\
  |x|^{\frac{b}{2}-\frac{2-\sigma}{p-1}}, & \text{if }b^2-4a<0.\\
\end{array}
\right.
$$
Since $c\geq 0$,
$$
\frac{b-\sqrt{b^{2}-4 a}}{2}-\frac{2-\sigma}{p-1}=-\dfrac{\sqrt{(n-2)^2-4c}}{2}-\frac{n-2}{2}\leq -\frac{n-2}{2},
$$
$\frac{b}{2}-\frac{2-\sigma}{p-1}=-\frac{n-2}{2}$. Thus $v$ may not belong to $H_{loc}^1(B_1)$ in this case.

\appendix

\section{Appendix: Classification of Positive Radially Symmetric Solutions}

In the appendix, we classify the positive solutions $\psi$ of \eqref{Equ-averaged-homo} as the following.
\begin{enumerate}[(i)]
  \item When $b=0, a<0$, $\psi$ either vanishes at infinity or is constant $(-a)^{\frac{1}{p-1}}$ or is a non-constant periodic function;
  \item When $b>0,a<0$, $\psi$ must converge to $(-a)^{\frac{1}{p-1}}$ as $t\rightarrow+\infty$;
  \item When $a\geq 0$, $\psi$ doesn't exists.
\end{enumerate}

$\bullet$
When $b=0,a<0$, equation \eqref{Equ-averaged-homo} reads
\begin{equation*}\psi''+a \psi+\psi^{p}=0, \quad-\infty<t<\infty,
\end{equation*}
with $1<p\leq\frac{n+2}{n-2}$. Thus
\begin{equation*}
\frac{\mathrm{d}}{\mathrm{d} t} H\left(\psi'(t), \psi(t)\right)=0,\quad\text{i.e.,}~
H\left(\psi'(t), \psi(t)\right)=h_{0}
\end{equation*}
for some constant $h_0$,
where $H$ is defined as in \eqref{Equ-def-H}.

Since
\begin{equation*}\begin{aligned}
\inf _{\mathbb{R} \times \mathbb{R}^{+}} H(\alpha, \beta) &=\inf _{\mathbb{R}^{+}}\left(a \cdot \beta^{2}+\frac{2}{p+1} \beta^{p+1}\right) \\
&=-\frac{p-1}{p+1}(-a)^{\frac{p+1}{p-1}},
\end{aligned}\end{equation*}
we have $h_0\geq -\frac{p-1}{p+1}(-a)^{\frac{p+1}{p-1}}$. As shown in Figure \ref{Pic-1} below,  solutions of \eqref{Equ-averaged-homo} with $b=0,a<0$ are classified by $h_0$ into the following cases.
\begin{figure}[htbp]
  \centering
  \includegraphics[width=0.5\linewidth]{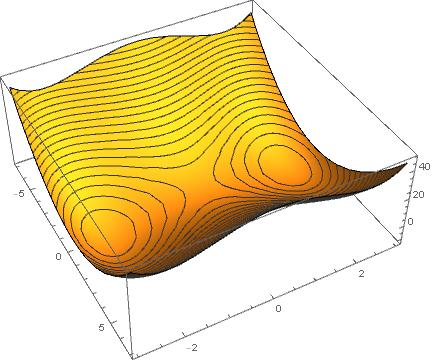}
  \caption{Picture of $H(\psi',\psi)$ with $a<0$.}\label{Pic-1}
\end{figure}
\begin{enumerate}
  \item $h_{0}=-\frac{p-1}{p+1}(-a)^{\frac{p+1}{p-1}}$. Then
  \begin{equation*}
   \psi(t) \equiv(-a)^{\frac{1}{p-1}},\quad\forall~t \in(-\infty, \infty).
  \end{equation*}
  \item $h_0=0$. Then by a direct computation, for some $\lambda>0$,
  \begin{equation*} \psi(t)=\left(\frac{\left(e^{-(p-1) \sqrt{-a} t}+\lambda^{2}\right) e^{\frac{p-1}{2} \sqrt{-a} t}}{\sqrt{-2 a(p+1) \lambda}}\right)^{-\frac{2}{p-1}}
  =O(1)\cdot e^{-\sqrt{-a}t}.
  \end{equation*}
  \item $h_0\in (-\frac{p-1}{p+1}(-a)^{\frac{p+1}{p-1}},0)$. Then we have a class of periodic solutions which are uniquely determined by $h_{0}.$ See more analysis of $\psi$ in Section 2.1 in \cite{KMPS}.
  \item $h_0>0$. Then the solution must become negative in finite time.
   Thus there are no positive solutions $\psi$ of \eqref{Equ-averaged-homo} with $b=0$ and $a<0$ such that $H(\psi',\psi)>0$.
\end{enumerate}

$\bullet$
When $b>0,a<0$, by  a direct computation,
\begin{equation}\label{equ-temp-12}
\frac{\mathrm{d}}{\mathrm{d} t} H\left(\psi'(t), \psi(t)\right)=-2 b \psi'^{2}(t) \leq 0.
\end{equation}
By a direct computation, \eqref{Equ-averaged-homo} with $b>0,a<0$ admits two non-negative constant solutions
$$
\psi\equiv 0\quad\text{and}\quad\psi\equiv (-a)^{1 /(p-1)}.
$$
Since $ H\left(\psi'(t), \psi(t)\right)$ is monotone decreasing and bounded from below, we let
$$
h_{-0}:=\lim_{t\rightarrow-\infty}H(\psi'(t),\psi(t))\quad\text{and}\quad
h_0:=\lim_{t\rightarrow+\infty}H(\psi'(t),\psi(t)),
$$
where $h_{-0}$ may be $+\infty$. We separate into the following cases by $h_{-0}$ and $h_0$.
\begin{enumerate}
  \item $h_{-0}>0$. Then there exists $N$ such that $H\left(\psi'(t), \psi(t)\right)>0$ for all $t \in(-\infty, N)$ and $\psi$ vanishes in finite time.

  In fact, if $0<h_{-0}<\infty$, then as shown in Figure \ref{Pic-1}, $\psi(t)$ is asymptotically periodic and becomes negative in finite time.
  It remains to rule out $h_{-0}=+\infty$ and $\psi$ always positive case.

  First we claim that for all $t<N$, $\psi'(t)<0$ always holds.
  By contradiction, we assume that for some $t_0<N$, $\psi'(t_0)=0, \psi(t_0)>0$ and $H(0,\psi(t_0))>0$. By \eqref{Equ-averaged-homo}, we have $\psi''(t_0)<0$ and thus $\psi'(t)<0$ in a left-neighbourhood $(t_0-2\epsilon,t_0)$ of $t_0$ for some $\epsilon>0$.
   As shown in Figure \ref{Pic-1}, $|\psi'|$ has a positive lower bound on $(-\infty,t_0-\epsilon]$, contradicting to $\psi>0$.

   However, since $\psi'(t)<0$ for all $t<N$, by \eqref{Equ-averaged-homo} again,
   there exists $N'\ll N$ such that $\psi''(t)<-1$ for all $t\in(-\infty,N']$. Thus this makes $\psi'(t)<0$ on $(-\infty,N)$ cannot happen.
   \item $h_{-0}=0, h_0=0$. Then $\psi\equiv 0$.
   \item $h_{-0}\leq 0, h_0<0$. Then
    $\psi(t)$ is asymptotically periodic to a positive periodic solution of \eqref{Equ-averaged-homo} with $b=0$. Thus by \eqref{Equ-averaged-homo}, $\psi''(t)$ is bounded. On the other hand,
   by \eqref{equ-temp-12}, $2b\psi'^2(t)$ is integrable on $[1,+\infty)$. Thus (see for instance the proof of \eqref{equ-limit-1der})
   $$
   \lim_{t\rightarrow+\infty}\psi'(t)=0\quad \text{and}\quad
   \lim_{t\rightarrow+\infty}\psi(t)=\psi(+\infty)
   $$
   for some constant $\psi(+\infty)$.
 By \eqref{Equ-averaged-homo} again, $\psi''(t)\rightarrow0$ as $t\rightarrow+\infty$ and thus $\psi(+\infty)$ is a constant solution of \eqref{Equ-averaged-homo}. Since $h_0<0$, the only possible case is
 $$
 \lim_{t\rightarrow+\infty}\psi(t)=
 (-a)^{1 /(p-1)}.
 $$
\end{enumerate}

$\bullet$
When $a\geq 0$, notice that
$$
\inf_{\mathbb R\times\mathbb R^+}H(\alpha,\beta)=\int_{\mathbb R}\left(a\cdot\beta^2+\frac{2}{p+1}\beta^{p+1}\right)=0.
$$
As shown in Figure \ref{Pic-2} and the calculus above,  the only possible case when $b=0$ is that $\psi(t)\equiv 0$, otherwise the solution would become negative in finite time. When $b>0$, similar to the discussion in $b>0,a<0$ and $h_{-0}>0$ case, $\psi$ also vanishes in finite time. Thus there are no positive smooth solutions when $a\geq 0$.

\begin{figure}[htbp]
  \centering
  \includegraphics[width=0.4\linewidth]{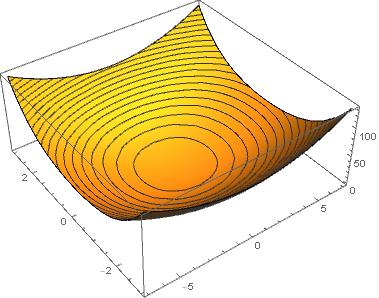}
  \caption{Picture of $H(\psi',\psi)$ with $a\geq 0$.}\label{Pic-2}
\end{figure}

Especially, Jin-Li-Xu \cite{15} proved the following radial symmetry of global solutions.
\begin{theorem}\label{Thm-symme-global}
 Let $f(x, t): \mathbb{R}^{n} \backslash\{0\} \times[0, \infty) \rightarrow[0, \infty)$ be a continuous
function such that for any $x \neq 0,0<\lambda<|x|,|z|>\lambda$ and $a \leq b,$ there holds
\begin{equation}\label{Equ-cond-symme}
\left(\frac{\lambda}{|z|}\right)^{n+2} f\left(x+\frac{\lambda^{2} z}{|z|^{2}},\left(\frac{|z|}{\lambda}\right)^{n-2} a\right)\leq f(x+z, b).
\end{equation}
If $v\in C^{2}\left(\mathbb{R}^{n} \backslash\{0\}\right)$ is a positive solution of
\begin{equation*}\Delta v+f(x, v)=0
\quad\text{in}~\mathbb R^n\setminus\{0\},
\end{equation*}
with $\limsup_{x\rightarrow 0}v(x)=+\infty$.
Then $v$ must be radially symmetric about the origin and $v^{\prime}(r)<0$ for all $0<r<\infty$.
\end{theorem}
By a direct computation, under assumption \eqref{Equ-cond-pucntured}, $f(x,t)=\dfrac{c}{|x|^2}t+\dfrac{t^p}{|x|^{\sigma}}$ satisfies condition \eqref{Equ-cond-symme}. Thus all positive smooth solutions of \eqref{Equ-Punctured} with \eqref{Equ-cond-pucntured} in $\mathbb R^n\setminus\{0\}$ must be radially symmetric if $0$ is not a $C^0$-removable singularity.
By transform \eqref{Equ-transform}, so are all positive smooth solutions of \eqref{Equ-Cylinder} with $$\limsup_{t\rightarrow+\infty}e^{\frac{n+b-2}{2}t}u(\theta,t)>0.$$

{
\section*{\small Acknowledgement}\small
 The authors would like to thank Professor Jingang Xiong for his guidance and support in this work.
}
\small

\bibliographystyle{plain}

\bibliography{Cylinder}

\noindent S. Chen \& Z. Liu

\medskip

\noindent  School of Mathematical Sciences, Beijing Normal University\\
Laboratory of Mathematics and Complex Systems, Ministry of Education\\
Beijing 100875, China \\[1mm]
Email: \\
\textsf{chenshan\_lyla@mail.bnu.edu.cn}\\
\textsf{liuzixiao@mail.bnu.edu.cn}

\end{document}